\documentclass[11pt]{article}

\usepackage{inputenc}
\usepackage{amsmath}
\usepackage{amsfonts}
\usepackage{amssymb}
\usepackage{graphicx}
\usepackage[T1]{fontenc}
\usepackage[english]{babel}
\usepackage{mathrsfs}
\usepackage{amsthm}
\usepackage[all]{xy}
\usepackage{geometry}
\usepackage{enumerate}
\allowdisplaybreaks
\usepackage{color}
\usepackage{appendix}
\usepackage{authblk}
\DeclareMathOperator{\sinc}{sinc}

\title{Analysis of the "Rolling carpet" strategy to eradicate an invasive species}
\author[1]{Luis Almeida}
\author[2]{Alexis L\'{e}culier}
\author[3]{Nicolas Vauchelet }

\affil[1]{LJLL - Sorbonne Universit\'{e} - \texttt{\small luis.almeida@sorbonne-universite.fr}}
\affil[2]{LJLL - Sorbonne Universit\'{e} - \texttt{\small alexis.leculier@sorbonne-universite.fr}}
\affil[3]{LAGA - Universit\'{e} Sorbonne Paris Nord - \texttt{\small vauchelet@math.univ-paris13.fr}}
\begin{document}                     
\maketitle

\begin{abstract}
In order to prevent the propagation of human diseases transmitted by mosquitoes (such as dengue or zika), one possible solution is to act directly on the mosquito population. In this work, we consider an invasive species (the mosquitoes) and we study two strategies to eradicate the population in the whole space by a local intervention. The dynamics of the population is modeled through a bistable reaction diffusion equation in an one dimensional setting and both strategies are based on the same idea : we act on a moving interval. The action of the first strategy is to kill as many individuals as we can in this moving interval. The action of the second strategy is to release sterile males in this moving interval. For both strategies, we manage to generate traveling waves that propagate in the opposite direction relative to the one of the natural invasive traveling wave. These cases correspond to succeeding in eradicating the invasive species. Furthermore, for the first strategy, we fully characterize the minimal size of the interval. All the results are illustrated by numerical simulations.
\end{abstract}

\textbf{Keywords}: 
Reaction-diffusion equations; population dynamics; comparison principle.

\textbf{AMS subject classifications}: 35K57, 92D25, 35C07.


\newtheorem{theorem}{Theorem}
\newtheorem{corollary}{Corollary}
\newtheorem{lemma}{Lemma}
\newtheorem{definition}{Definition}
\newtheorem{proposition}{Proposition}
\newtheorem{example}{Example}
\newtheorem{definition-proposition}{Definition-Proposition}
\newtheorem*{notation}{Notation}
\newtheorem*{remark}{Remark}

\section{Introduction}

In this article, we focus on two strategies to eradicate a naturally invading species by a local action on a moving frame. Both strategies are based on the same idea of the "rolling carpet": we act on an interval and we move this interval with a constant speed from an empty area toward an invaded area. The main difference between both strategies is the employed action on this moving interval. The first strategy is quite simple: we kill as many individuals as we can in the interval. The second strategy is more evolved and consists in using the sterile insect technique, i.e. sterile males are released in the interval. The aim of this work is to investigate the following questions :
\begin{itemize}
    \item Do the strategies work ?
    \item What can be said about the minimal size of the interval ? 
    \item What can be said on the speed of the rolling carpet ? 
\end{itemize}

\vspace{0.2cm}

\subsection{The mathematical models}

As explained above, both strategies rely on the same idea. Let $v$ be a population density (which takes into account  or not the sex of the individuals) that we suppose to live in one-dimensional space. The dynamics of the population is driven by the following reaction-diffusion equation:
\begin{equation}
\partial_t v  - \partial_{xx} v = g(v). 
\end{equation}
In order to model an \textit{Allee effect} (a small density will go to extinction whereas a large number of individuals in a large space will invade the territory), we assume that the reaction term is bistable and normalizes such that $0$ and $1$ are the two stable steady state. Reaction terms will be specified later. We will also assume that at the initial time, the population is "almost" established in a neighbourhood of $-\infty$ and "almost" absent in a neighbourhood of $+\infty$: there exists $\lambda_0>0$ such that 
\begin{equation}\label{H2S1'} \tag{H0}
 \forall x>M, \qquad    (1-e^{-\lambda_0 x}) \leq  v(-x, 0)< 1  \quad \text{ and } \quad 0 <  v(x, 0) < e^{-\lambda_0 x}.
\end{equation}
Then, it is well known that the solution $v$ adopts the same shape than the traveling wave which connects the two stable states of $g$, $0$ and $1$, with a constant speed $c_{bistable}$. Moreover, the sign of $c_{bistable}$ is determined by the sign of $\int_{0}^1 g(s) ds$ (see e.g.  \cite[Theorem 4.9]{Perthame}). Therefore, if $c_{bistable} \leq 0$, the intrinsic dynamics avoids invasion phenomena by itself. However if $c_{bistable}>0$, there exists an invasion phenomenon: the native population will tend to occupy the whole space. Our work will be done in the setting $\int_0^1 g(s)\,ds >0$, since it concerns strategy to push the population towards $-\infty$.

For both strategies, we act on a moving interval $(ct , L + ct)$ with $c<0$. Somehow, the dynamics is driven by 
\[ \partial_t u - \partial_{xx} u = g(u) 1_{ \left\lbrace x < ct , \  x>L+ct \right\rbrace} + Act(x,u) 1_{\left\lbrace ct \leq x \leq L+ct \right\rbrace}.\]
The aim of this paper is to prove that for well-designed functions $Act$ there exist traveling wave solutions that connect $0$ to $1$ with a negative speed $c$. 
Such traveling wave solutions satisfy the equation
\begin{equation}\label{P'} \tag{$\mathcal{P}$}
 \left\lbrace
\begin{aligned}
    & -cu'- u'' = g(u) 1_{ \left\lbrace x < 0, \  x>L \right\rbrace} + Act(x,u) 1_{\left\lbrace 0 < x < L \right\rbrace}, \\
    &u(-\infty) = 1, \quad u(+\infty) = 0.
\end{aligned}
\right. 
\end{equation}
When such a solution $u$ exists with a negative speed $c$, then this strategy allows to eradicate the species.
Indeed, this traveling wave solution will be a super-solution above $v(t=0)$, the comparison principle implies that $0 \leq v \leq u(x + ct)$. Since, $\underset{ t \to +\infty}{\lim} u(x+ct) = 0$ for any $x\in\mathbb{R}$, we conclude that the population $v$ goes to extinction as time grows.


\subsubsection{The killing strategy}

In this first strategy, we simply consider a population modelled by its density $u$ whose dynamics is governed by the reaction-diffusion equation
\[ \partial_t u - \partial_{xx} u = g(u),\]
with $g$ a "classical" smooth bistable reaction term:
\begin{equation} \label{H1S1} \tag{H1}
\begin{aligned}
&g(0) = g(1) = g(\alpha) = 0, \\
&g'(0)< 0, \quad g'(1)<0 , \quad g'(\alpha)>0 \quad
\text{ and } \quad \int_0^1 g(u)du >0. 
\end{aligned}
\end{equation}
Since we have assumed that $\int_0^1 g(u) du  >0$, there exists $\beta \in (\alpha , 1)$ such that
\begin{equation}\label{beta}
\int_0^\beta g(u)du =  0.
\end{equation}
We also introduce $\alpha_1$ and $\alpha_2$ such that 
\begin{equation}\label{alpha12}
\alpha_1 < \alpha < \alpha_2 \quad \text{ and } \quad g'(\alpha_1) = g'(\alpha_2) = 0. 
\end{equation}
The strategy is quite simple: on the moving interval $(ct,L+ct)$, we kill individuals with a given rate. Therefore, the reaction term in this interval is replaced by a simple death term $-\mu u$. The new dynamics is driven by the equation 
\[ \partial_t u - \partial_{xx} u = g(u) 1_{\left\lbrace x<ct, \  x>L+ct \right\rbrace} - \mu 1_{ \left\lbrace ct \leq x \leq L+ct \right\rbrace }u. \]
In this interval, the individuals die "naturally" and are "killed", hence it is natural to assume that 
\begin{equation}
    \label{H2S1}\tag{H2}
    -\mu u \leq  g(u). 
\end{equation}
Then, as explained above, the goal is to prove the existence of a traveling wave solution of
\begin{equation}\label{PS1} \tag{$\mathcal{P}_1$}
    \left\lbrace
    \begin{aligned}
    &-c u ' - u'' = g(u)1_{ \left\lbrace x<0, \ x> L \right\rbrace} - \mu 1_{ \left\lbrace 0  <  x <  L \right\rbrace }u, \\
    &u(-\infty) = 1, \quad u(+\infty) = 0,
    \end{aligned}
    \right.
\end{equation}
with $c<0$.
Since equation \eqref{PS1} depends on two parameters : the speed $c$ and the size $L$, we will always specify both of them.

\subsubsection{The sterile males strategy}

This strategy is more evolved. The idea is to release artificially sterilized males in the interval $(ct, L+ct)$. As we can expect, the main consequence of the releases is the decline of the growth rate. For this strategy, the equation must take into account the proportion of males $m$, females $f$ and sterile males $m_S$ : $v = f + m + m_S$. A classical assumption used to simplify the system is to consider that the proportion of fertile males and females are equivalent ($m \sim f$). Therefore, we focus only on the proportion of fertile females and sterile males. In the case where we suppose that both the females and the sterile males diffuse, the equations which drive the dynamics reads
\begin{equation} 
\left\lbrace
\begin{aligned}
& \partial_t f - f'' = g(f, m_S)\\
& \partial_t m_S - m_S'' = \mathrm{M} 1_{\left\lbrace ct < x < L +ct \right\rbrace} - \mu_s m_S,
\end{aligned}
\right.
\end{equation}
where $M$ is the constant density of released males in the moving frame.
The reaction term $g$ is such that without any sterile males, the density of the population follows a bistable dynamics: 
\begin{equation}\label{H3S2}\tag{H3}
    g(f, 0) \text{ satisfies assumption \eqref{H1S1}.} 
\end{equation}
We also assume that as we introduce sterile males, the birth rate decreases
\begin{equation}\label{H4PS2} \tag{H4}
 \dfrac{\partial g(f,m)}{\partial m} < 0 \quad \text{ and  } \quad    g( f, m ) \underset{ m \to + \infty}{\longrightarrow } -\mu f .
\end{equation}
We add a technical assumption and a natural one
\begin{equation}\label{H5PS2} \tag{H5}
\begin{aligned}
    g \text{ is uniformly continuous  with respect to }m 
   \qquad  \text{ and } \qquad  g(0, m)  = 0.
\end{aligned}
\end{equation}
In the next section, we provide a specific example of such a function $g$. We look for traveling wave solutions with $c<0$ of 
\begin{equation}  \label{PS2}  \tag{$\mathcal{P}_2$}
\left\lbrace
\begin{aligned}
& -c f' - f'' =g(f, m_S)\\
& -c m_S' - m_S'' = \mathrm{M} 1_{\left\lbrace 0 < x < L \right\rbrace} - \mu_s m_S, \\
&f(-\infty) = 1, \quad f(\infty) = 0.
\end{aligned}
\right.
\end{equation}
Since equation \eqref{PS2} depends on three parameters : the speed $c$, the size $L$ and the quantity of released sterile males $\mathrm{M}$, we will always refer to \eqref{PS2} by specifying these three parameters $(c, L, \mathrm{M})$.

\subsection{Biological motivations}\label{BiologicalMotivation}

The "rolling carpet" strategies have been already used in the field to eradicate some species of insect, for instance the tsetse fly in large area \cite{Vreysen2007}.
A motivation of this work is to fight against the spreading of mosquitoes born diseases which cause more than 700 000 deaths annually according the World Health Organization \cite{who}. Indeed, such insects are the vectors of many lethal diseases as the dengue, Malaria and others. Both strategies are already used in practice \cite{SterileInsectBook, BiblioSterileMales}. 

In this paper, we propose to analyze from a mathematical point of view, these strategies by studying conditions that make sure the succeed. The first strategy (the killing strategy) is usually achieved by spreading insecticide. The insecticide has two disadvantages. First, it does not target a specific type of mosquitoes, nor even mosquito relative other insects (having possible side effect on the whole ecosystem). Secondly, the mosquitoes may adapt and become more resistant to the insecticide (see e.g. \cite{Resistance}). Therefore the second strategy seems to be more advantageous and environmental friendly.

One possible model of the reaction term for the mosquito dynamics is the following:
\begin{equation}
    \label{gms}
    g(f,m) = \frac{C_1 f^2(1-e^{-(\beta_1 f + \beta_2 m)})}{f^2(1-e^{-(\beta_1 f + \beta_2 m)} )+ C_2 (\beta_1 f + \beta_2 m)} - \mu f.
\end{equation}
where $C_{1,2}$, $\beta_{1,2}$ and $\mu$ are positive constants which depend on several intrinsic constants (such as death and birth rates at different stage of life, the environmental capacity...).
Such reaction term can be obtained by making a quasi-stationary assumption in the complete model introduced in \cite{strugarek_TIS} (see also \cite{almeida2020sterile} and the references therein for a precise derivation of this model). Notice that this reaction term is a simplification of a more evolved system which takes into account different stages of development of the mosquitoes (egg, larva, adult...). We also underline that for a suitable choice, relevant from a biological point of view, of constants $C_{1,2}$, $\beta_{1,2}$ and $\mu$, this reaction term verifies the hypothesis \eqref{H3S2}, \eqref{H4PS2} and \eqref{H5PS2}.


\vspace{0.25cm}

Notice that we describe the particular example of mosquitoes but this work can be adapted to other invasive species and to other frameworks as social sciences (see \cite{BerRodRyz} for example).

\subsection{Main results and comments}

For both strategies, we prove the existence of a traveling wave with negative speed $c<0$. For the first equation \eqref{PS1}, our main result reads
\begin{theorem}\label{mainPS1}
Under the assumptions \eqref{H1S1} and \eqref{H2S1}, there exists $\Lambda_0>0$ and a decreasing bijection 
\begin{equation}
\begin{aligned}
    \Lambda :  \ ]-\infty, 0]& \rightarrow [\Lambda_0 , +\infty[ \\
    c &\mapsto \Lambda(c) 
    \end{aligned}
\end{equation}
such that 
\begin{enumerate}
    \item For any speed $c\leq 0$ and $L<\Lambda(c)$, \eqref{PS1} with parameters $(c, L)$ does not admit a solution. Moreover, the equation 
    \[ \left\lbrace
    \begin{aligned}
    &-cu' - u'' = g(u) 1_{\left\lbrace x < 0, \ x > L \right\rbrace}(x) - \mu 1_{\left\lbrace 0 <x < L \right\rbrace} \\
    &u(-\infty) = 1.
    \end{aligned}
    \right.\]
     verifies $u(+\infty)= 1$.
   \item For any speed $c\leq 0$ and $L>\Lambda(c)$, \eqref{PS1} with parameters $(c, L)$ admits a decreasing solution. 
\end{enumerate}
Moreover, when $L=\Lambda(c)$, if we assume that $g$ is convex in the interval $(0, \alpha)$ then
\begin{enumerate}
   \item if $c$ is such that $- 2\sqrt{g'(\alpha)} < c \leq 0$, \eqref{PS1} with parameters $(c, \Lambda(c))$ admits a solution. This solution satisfies $u'(\Lambda(c)) = 0$ and $\alpha < u(\Lambda(c)) \leq \beta $,
   \item if $c$ is such that $c\leq -2\sqrt{g'(\alpha)}$, \eqref{PS1} with parameters $(c, \Lambda(c))$ does not admit a solution. Moreover, we have that $\underset{L> \Lambda(c)}{\sup} u(L)  = \alpha$. 
\end{enumerate}
\end{theorem}

For the second strategy, the main result reads
\begin{theorem}\label{mainPS2}
Under the assumptions \eqref{H3S2}, \eqref{H4PS2} and \eqref{H5PS2}, there exists a function 
\[  \Pi: (c, L) \in \mathbb{R}^- \times \mathbb{R}^{+*} \mapsto \Pi(c, L)\]
such that for any speed $c \leq 0$ and size $L>0$
\begin{enumerate}
    \item For any $\mathrm{M}> \Pi(c,L)$, \eqref{PS2} with parameters $(c, L, M)$ admits a solution, 
    \item For any $\mathrm{M}< \Pi(c,L)$, \eqref{PS2} with parameters $(c,L,M)$ does not admit a solution.
\end{enumerate}
Moreover, for a fixed speed $c<0$, we have that
\[ \underset{ L \to 0}{\lim }  \ \Pi(c, L) = +\infty \quad \text{ and } \quad  \underset{ L \to +\infty}{\lim } \Pi(c, L) >\Pi_\infty(c)>0.\]
What's more, for a fixed size $L>0$, 
\[ \underset{c \to -\infty}{\lim} \ \Pi(c, L) = +\infty.\]
\end{theorem}

For both strategies, we have succeeded in generating a traveling wave for any (negative) speed which goes in the opposite sense than the unique "natural" traveling wave. In both cases, the proofs are based on the construction of a sub-solution $\phi_-$ and a super-solution $\phi_+$ to \eqref{P'} that are right-ordered (i.e. $\phi_- < \phi_+$). However, each strategy has its own technical difficulties and therefore, we present each strategy separately. We introduce all the necessary tools for each strategy in the corresponding sections. \\
We underline that since the first equation is easier to work with, we have obtained a complete description of the traveling waves, in particular we know what happens for the critical case $L=\Lambda(c)$. We also emphasize that we did not expect to obtain the dichotomy $c> - 2\sqrt{g'(\alpha)} $ and $c\leq - 2\sqrt{g'(\alpha)}$. It relies on the fact that in a neighbourhood of $\pm \infty$, we can understand the equation as an autonomous equation. Using this, the tails at $\pm\infty$ of the traveling waves are unique. Moreover, for $c\leq - 2\sqrt{g'(\alpha)}$, there exists a traveling wave $u_{\mathrm{KPP}}$ which connects $0$ to $\alpha$ (solution of a Fisher-KPP type equation). With these two remarks in mind, we prove that for $c\leq - 2\sqrt{g'(\alpha)}$, the tails of the traveling waves at $+\infty$ behave like $u_{\mathrm{KPP}}$. Without the technical assumption $g$ convex into $(0, \alpha)$, we have to proceed case by case. However, the strategy seems to be robust if we know the existence of traveling waves which connect $0$ to $\alpha$ with a negative speed. \\
Contrary to the first equation, the second equation is more difficult to work with because it is a fully non-autonomous system. Indeed, $m_S$ has its own dynamics (independent of $f$) but the sterile males spread on the whole domain $\mathbb{R}$. Even if we \textit{act} only on a small part of the domain, this spreading makes the equation on $f$ fully non-autonomous (even by parts).

\subsection{State of the art}

Bistable equations to model propagation phenomena with an \textit{Allee effect} were initially introduced in the pioneer work \cite{aronsonweinber}, where the existence of a traveling wave solution which connects the stable states $0$ and $1$ is established. Since this work, plenty of works study variations of this problem as \cite{TW, TW2}. \\
To our knowledge the mathematical study of an action on a small moving interval to eradicate an invasive population of invasive species has not been addressed. The previous works focused on the specific case $c=0$. In this specific case, the idea is not to eradicate the population but block the front propagation. One of the first articles which focused on this kind of mathematical question is \cite{LewisKeener} (see also \cite{Pauw}). In this paper, the authors assume that in an interval the reaction term is $0$. Using a phase plane analysis, they prove the existence of a blocking if the size of this interval is large enough. We mention also the papers \cite{Chapuisat,NadinStrugarek} where blocking in biological systems is analyzed. In a different setting, a blocking strategy is studied in \cite{BerRodRyz}. The framework of this article is social sciences. The authors investigate the employment of finite resources to prevent the invasion of criminal activity. This latter article was adapted by one of its authors \cite{Rodriguez} to prevent invasion phenomena in biology. We underline that in \cite{Rodriguez}, even if the author considers sterile male releases, the article did not take into account the spread of such sterile males. Therefore, from a mathematical point of view, these models are closer to our first strategy than the second one. The second strategy with $c=0$, i.e. the sterile insect technique used as a barrier to block re-infestation is studied by two of the authors in \cite{almeida2020sterile}.
By taking $c=0$, we recover the main results of \cite{BerRodRyz}, \cite{Rodriguez} and \cite{almeida2020sterile}. Finally, we quote \cite{Opt-Trelat-Zhu-Zua1, Opt-Trelat-Zhu-Zua2} which deal also with the sterile males method to prevent invasion. In these two articles, the authors focus on an optimal problem. The aim is to optimize the releases of sterile males along time in order to minimize a cost-function. This cost-function takes advantage of the existence of traveling wave solutions which drive the solution to $0$. The main difference with our work is that we do not investigate any variation of the releases in time whereas the authors do not consider any spatial structure on the releases of sterile males. 

\subsection{Outline of the paper}

In section \ref{SectionPS1}, we focus on the killing strategy and we prove Theorem \ref{mainPS1}. Section \ref{SectionPS2} is devoted to the sterile male strategy and the proof of Theorem \ref{mainPS2}. We present numerical results that illustrate our results in section \ref{SectionNum} for both strategies. The second strategy is simulated for the application that we have introduced in section \ref{BiologicalMotivation}. Finally, we end this article with a conclusion and some perspectives in section \ref{SectionConclusion}.

\section{Study of the killing strategy} \label{SectionPS1}

In a first subsection, we introduce all the definitions, tools and intermediate results needed in the proof of Theorem \ref{mainPS1}. Subsections \ref{subsec:proof1} and \ref{subsec:proof2} are devoted to the proof of the first two points of Theorem \ref{mainPS1}. The rest of this section focuses on the critical case $L= \Lambda(c)$.

\subsection{Intermediate results}

We detail here some definitions and intermediate results needed in the proof of Theorem \ref{mainPS1}. The proofs of these results are postponed to subsection \ref{subsec:critiq}.

We begin by establishing that there exists a set of parameters such that \eqref{PS1} admits a solution. 

\begin{proposition}\label{intermediate1}
For any speed $c<0$, there exists a size $L_0$ such that \eqref{PS1} admits a solution for parameters $(c, L)$ with $L>L_0$. 
\end{proposition}

The basic idea is to prove that there exists a sub-solution $\psi_-$ and a super-solution $\psi_+$ such that 
\begin{equation}\label{sub-super}
    \psi_- \leq \psi_+, \quad \psi_-(-\infty) = \psi_+(-\infty) = 1 \quad \text{ and } \quad \psi_-(+\infty) = \psi_+(+\infty) = 0.
\end{equation}

Since, the reaction term is singular for  $x \in \left\lbrace 0, L \right\rbrace$, we recall the definitions of sub- (resp. super-) solutions (see \cite{Pauw}):

\begin{definition}[Sub- and super-solution]
A function $\psi_- \in C^2(\mathbb{R} \backslash \left\lbrace 0, L \right\rbrace )$ is a sub-solution to \eqref{PS1} if it satisfies 
\[ -c \psi_-' - \psi_-'' \leq g(\psi_-) 1_{]0,L[^c} - \psi_- 1_{]0,L[} \quad \text{ and } \quad \underset{ x < \xi}{\underset{ x \to \xi}{\lim}} \psi_-'(x) \leq \underset{x > \xi}{\underset{ x \to \xi}{\lim}} \psi_-'(x) \  \text{ for } \  \xi \in \left\lbrace 0,L \right\rbrace\]
(resp. $\psi_+$ is a super solution if it satisfies the reverse inequalities than above). 
\end{definition}

Since, in general, the solutions of \eqref{PS1} might be not unique, we need to select a "proper" solution. With this in mind, we define a solution as follows:
\begin{definition}\label{intermediate4}
We define $u$ as a solution of \eqref{PS1} with parameters $(c,L)$ as the supremum of the sub-solutions of equation \eqref{PS1}
\[ i.e. \qquad u(x) = \sup \left\lbrace \psi_-(x), \quad \text{ with } \psi_- \text{ a sub-solution of } \eqref{PS1} \right\rbrace.\]
\end{definition}
 
This solution is well defined according to the construction by the sub- and super-solution technique (see \cite{smoller}). We recall that $u$ is the minimal non-trivial solution. In other words, if $u$ is the solution of \eqref{PS1} with parameters $(c,L)$ and $v$ is an other solution which satisfies $0 \leq v \leq u$ in $\mathbb{R}^+$ then we deduce that $v=0$ or $v=u$. Then, from any solution of \eqref{PS1} with parameters $(c,L)$, we can construct a solution of \eqref{PS1} with parameters $(c, L')$ and $L'>L$ thanks to 
\begin{lemma}\label{intermediate3}
Let $(c,L_*)$ be a set of parameters such that \eqref{PS1} admits a solution $\overline{u}$. Then for any $L>L_*$, $\overline{u}$ is a super-solution to \eqref{PS1} with the set of parameters $(c, L)$.
\end{lemma}

The set of solutions are naturally ordered with respect to $L$. Moreover, all the solutions are decreasing. We sum up these two last results into the following Proposition:

\begin{proposition}\label{intermediate5}
The following assertions hold true:
\begin{enumerate}
    \item Let $L_1<L_2$ be such that there exists $u_1$, $u_2$ two solutions of \eqref{PS1} with parameters $(c,L_1)$ and $(c, L_2)$. Then, we have $u_2 \leq u_1$. 
    \item If there exists a solution $u$ of \eqref{PS1} with parameters $(c,L)$ then $u$ is decreasing.
\end{enumerate}
\end{proposition}

Now that we know that the set of solutions of \eqref{PS1} (where $L$ is seen as a free parameter) is ordered, one can introduce
\begin{equation}
    \label{lambda}
    \Lambda(c)  = \inf \left\lbrace  L>0, \quad  \text{there exists a solution to \eqref{PS1} with parameters } (c,L) \right\rbrace.
\end{equation}

In the next proposition, we provide the main properties of $\Lambda$. 

\begin{proposition}\label{intermdiate6}
The function $\Lambda$ is well defined for all $c<0$. Moreover, the following assertions hold true:
\begin{enumerate}
    \item There exists $\Lambda_0>0$ such that 
\[ \Lambda_0 \leq \underset{c \in ]-\infty, 0]}{\inf} \Lambda(c). \]
    \item The function $\Lambda$ is decreasing with respect to $c$. 
    \item $\underset{c \to -\infty}{\lim} \Lambda(c) =  +\infty$,
    \item For $L = \Lambda(c)$, we distinguish two cases:
    \begin{description}
    \item [Case 1 :]  There does not exist a decreasing traveling wave which connects $0$ and $\alpha$ with speed $c$. Then \eqref{PS1} with parameters $(c, \Lambda(c))$ admits a solution $u$ and we have $\alpha \leq u(\Lambda(c)) \leq \beta$ and $u'(\Lambda(c))= 0$.
    \item [Case 2 :]  There exists a decreasing traveling wave $u_{\mathrm{TW}}$ which connects the unstable state $\alpha$ and the stable state $0$ with speed $c$, i.e. a solution of
    \begin{equation}\label{TWalpha}
    \left\lbrace
    \begin{aligned}
        &-c u_{\mathrm{TW}}' - u_{\mathrm{TW}}'' = g(u_{\mathrm{TW}}), \\
        &u_{\mathrm{TW}}(-\infty) = \alpha \quad \text{ and } \quad u_{\mathrm{TW}}(+\infty) = 0.
    \end{aligned}
    \right.
    \end{equation}
    Then \eqref{PS1} with parameters $(c, \Lambda(c))$ does not admit a solution. Moreover, we have
    \[ \underset{ L \to \Lambda(c)^+}{\lim} u_L(L) = \alpha \quad \text{ and } \quad \underset{L  \to \Lambda(c)^+}{\lim}u'_L(L) = 0.\]
    \end{description}
\end{enumerate}
\end{proposition}

We underline that the last assertion will be useful to characterize numerically $\Lambda(c)$. Notice that the previous result is more general than Theorem \ref{mainPS1}. We recover the last statements of Theorem \ref{mainPS1} thanks to the following Corollary 

\begin{corollary}\label{Corollarygconv}
If $g$ is convex in $(0, \alpha)$ then \eqref{PS1} with parameters $(c, \Lambda(c))$ admits a solution if and only if $-2 \sqrt{g'(\alpha)} < c \leq 0$. 
\end{corollary}

We notice that it is the only statement where we have used the hypothesis $g$ is convex in $(0, \alpha)$. Furthermore, we underline that according to \cite{BerRodRyz}, if $u$ is a solution of \eqref{PS1} with parameters $(0, \Lambda(0))$, it follows that
\[ u(\Lambda(0) ) = \beta \quad \text{ and } \quad u'(\Lambda(0)) = 0.\]
We finish with a last proposition which characterizes the case $L < \Lambda(c)$. 

\begin{proposition}\label{intermediate7}
For any $L< \Lambda(c)$, there does not exists a solution of\eqref{PS1} with parameters $(c,L)$. Let $u$ be the solution of 
\[\left\lbrace 
\begin{aligned}
&-cu' - u'' = g(u) 1_{\left\lbrace x<0 , \ x>L  \right\rbrace} - \mu u 1_{\left\lbrace 0 < x  < L \right\rbrace}, \\
&u(-\infty) = 1
\end{aligned}
\right.\]
then we have $u(+\infty) = 1$. Moreover, there exists a unique $x_0 \in ]0,L[$ such that $u'(x) = 0$. 
\end{proposition}


\subsection{Proof of the first part of Theorem \ref{mainPS1}}\label{subsec:proof1}

The first point of Theorem \ref{mainPS1} is a direct consequence of the existence of $\Lambda(c)$, defined by \eqref{lambda}, and Proposition \ref{intermediate7}.
The second one follows from the definition of $\Lambda(c)$ and Lemma \ref{intermediate3}.
The last two points are direct applications of Proposition \ref{intermdiate6} and Corollary \ref{Corollarygconv}. 

\begin{remark} 
We highlight that the proof works because we define a solution of \eqref{PS1} as the supremum of the sub-solutions. The setting of Theorem \ref{mainPS1} could be false if we consider another type of definition. For instance, in \cite{BerRodRyz} the authors construct a solution of 
\[\left\lbrace 
\begin{aligned}
& - u'' = g(u) 1_{\left\lbrace x<0 , \ x>L  \right\rbrace} - \mu u 1_{\left\lbrace 0 < x  < L \right\rbrace}, \\
&u(-\infty) = 1, \quad u(+\infty) = 1
\end{aligned}
\right.\]
(a similar equation than \eqref{PS1} with parameters $(0, L)$) which satisfies $u( \pm \infty ) = 1$ for any $L>0$. Obviously, this solution is greater than any minimal solution that decreases to $0$ at $+\infty$ (when it does exist). 
\end{remark}

\subsection{Construction of a solution}\label{subsec:proof2}

In this part, we fix $c<0$ and we will show the existence of $L>0$ such that \eqref{PS1} with parameters $(c, L)$ admits a sub-solution and a super-solution that satisfy \eqref{sub-super}. In a first part, we construct a super-solution, then we construct a sub-solution, and finally, we conclude to the existence of a solution.

\subsubsection{Construction of the super-solution } We split the construction of the super-solution into two lemmas: in the first one we describe the super-solution on the interval $(0,L)$, in the second one we describe the solution on the interval $(L, +\infty)$. As we will see later on, the super-solution is simply constant equal to $1$ on $(-\infty, 0)$. 

\begin{lemma}\label{lemmasuper1}
For any $\gamma > 0$, there exists a size $L>0$ such that there exists a positive solution $v_1$ of the following problem
\begin{equation}\label{eqsuper1}
\left\lbrace
\begin{aligned}
    &-c v_1' - v_1'' = - \mu v_1 \\
    &v_1(0) = 1, \quad v_1(L) = \gamma, \\
    &v_1'(0) \leq 0, \quad v_1'(L) =0.
\end{aligned}
\right.
\end{equation}
\end{lemma}

\begin{proof} 
According to the limits at $L$, it is natural to search a solution on the form:
\[ v_1(x)  = \frac{\gamma}{\lambda_+ - \lambda_- } [\lambda_+ e^{\lambda_- (x-L)} - \lambda_- e^{\lambda_+(x-L)}] \]
with $\lambda_{\pm}$ the negative and the positive roots of $r^2 +cr - \mu$. Next, we look for a size $L$ such that the conditions at $0$ are satisfied. We write the expression of $v_1(0)$ and $v_1'(0)$ as follows:
\begin{align*} 
&v_1(0) = \frac{\gamma e^{\frac{cL}{2}} }{\lambda_+ - \lambda_- } \left[  -c \sinh(\frac{\sqrt{\Delta}L}{2} ) + \sqrt{\Delta} \cosh(\frac{\sqrt{\Delta} L}{2}) \right] : =  \frac{\gamma }{\lambda_+ - \lambda_- }  \psi_1\left(\frac{L}{2} \right) \\
\text{ and } \quad & v_1'(0)  = \frac{2 \gamma \lambda_+ \lambda_- e^{\frac{cL}{2}} \sinh (\frac{\sqrt{\Delta} L}{2})}{\lambda_+ - \lambda_-} : =   \frac{2\gamma \lambda_+\lambda_- }{\lambda_+ - \lambda_- }  \psi_2 \left(\frac{L}{2} \right).
\end{align*}
where $\Delta = c^2 + 4 \mu$, $\psi_1 (L) =  e^{cL}[-c\sinh(\sqrt{\Delta} L) + \sqrt{\Delta} \cosh(\sqrt{\Delta} L) ]$ and $\psi_2 (L) = e^{cL} \sinh(\sqrt{\Delta} L )$. First, notice that the condition $v_1'(0)<0$ is trivially satisfied for any $L>0$. Remarking that $\psi_1(0) = \sqrt{\Delta}$ and $\underset{ L \to +\infty}{\lim} \psi_1(L) = + \infty$, we conclude to the existence of $L>0$ such that \eqref{eqsuper1} admits a solution.
\end{proof}

\begin{remark}
We underline that we have first fixed the speed $c$ and next the size $L$ which in turn depends on $c$. One can remark that as $|c|$ increases, one has that $L$ increases too.
\end{remark}

\begin{lemma}\label{Lemmasuper2}
There exists $\gamma_0 \in ]0, \alpha_1[$ and $\delta_0>0$ such that the solution of the following ODE
\begin{equation}\label{PS1super1}
\left\lbrace
\begin{aligned}
&-c v_2' - v_2'' = g(v_2),     \\
&v_2(L) = \gamma_0, \quad v_2'(L) =-\delta_0
\end{aligned}
\right.
\end{equation}
satisfies 
\[v_2(+\infty)  = 0 \quad \text{ and } \quad v_2'(+\infty) = 0.\]
\end{lemma}

\begin{proof}
The proof follows the application of the stable manifold theorem. Indeed, the equilibrium $(0,0)$ of 
\[ \left\lbrace 
\begin{aligned}
&v_2'=w_2, \\
&w_2' = -c w_2 - g(v_2),
\end{aligned}
\right.\]
is a saddle point. Moreover, the stable tangent space is generated by the vector $\left(1, \frac{|c| - \sqrt{c^2 + 4 |g'(0)|}}{2} \right)$. The conclusion follows. 
\end{proof}

\begin{remark}
By a more thorough analysis, we can prove that a sufficient condition is $v_2'(L)<0$. The interested reader can follow the last part of the proof of Proposition \ref{intermediate5} in the subsection \ref{monotonicityresults}. Since the proof above is quite simple and sufficient for the content of this section, we chose to present this one. 
\end{remark}

\begin{proposition}\label{proposuperPS1}
There exists a size $L>0$ such that there exists a super-solution $\psi_+$ to \eqref{PS1} with parameters $(c,L)$ which satisfies 
\[\underset{ x \to -\infty}{\lim } \psi_+(x) = 1 \quad  \text{ and } \quad \underset{ x \to + \infty }{\lim} \psi_+(x) = 0.\]
\end{proposition}

\begin{proof}
First, we take $c<0$. Then, we fix $\gamma_0$ like in Lemma \ref{Lemmasuper2}. Next, take $L>0$ provided by Lemma \ref{lemmasuper1}. Finally, we define 
\[ \psi_+(x)  = \left\lbrace
\begin{aligned}
& 1 && \text{ for } x \in (-\infty,0), \\
&v_1(x) &&\text{ for } x \in (0, L), \\
&v_2(x)  &&\text{ for } x \in (L, +\infty)
\end{aligned}
\right.\]
(where $v_1,\ v_2$ are provided by Lemmas \ref{lemmasuper1} and \ref{Lemmasuper2}). It is trivial that $\psi_+$ is a super solution in each interval $(-\infty, 0), \  (0, L)$ and $(L, +\infty)$. We only have to check the compatibility condition of the derivative at $\left\lbrace 0, L \right\rbrace$ which are satisfied according to Lemmas \ref{lemmasuper1} and \ref{Lemmasuper2}. \\
The limits at $\pm \infty$ hold true by definition of $\psi_+$. 
\end{proof}

\subsubsection{Construction of a sub-solution } We construct directly a sub-solution since there is no difficulty to obtain it. 

\begin{proposition}\label{proposubPS1}
There exists a sub-solution $\psi_-$ to \eqref{PS1} such that 
\[ \underset{x \to -\infty}{\lim} \psi_-(x) = 1 \quad \text{ and } \quad \psi_-(x) = 0 \text{ for } x >0.\]
\end{proposition}

\begin{proof}
We construct this sub-solution piecewise. Let $\chi $ be the decreasing solution of 
\[ \left\lbrace \begin{aligned}
& -\chi'' = g(\chi), \\
& \chi(-\infty ) =  1, \quad \chi(0) = 0.
\end{aligned}
\right.\]
Such a solution exists since $\int_0^1 g(u)du > 0$. Notice that $\chi'(0)  = -\sqrt{2\int_0^1 g(u)du} < 0$. The construction is classical and relies on a phase plane analysis as the one developped in \cite{BerRodRyz}. Therefore, we let it for the interested reader. Next, we extend $\chi$ by $0$ on $\mathbb{R}_+$ in order to provide a sub-solution. 
\end{proof}

\subsubsection{Conclusion : Construction of the solution } We construct a solution from the above sub- and super-solutions. 

\begin{proof}[Proof of Proposition \ref{intermediate1}]
According to Propositions \ref{proposuperPS1} and \ref{proposubPS1}, there exists a sub- and a super-solution that are well-ordered. By applying the classical technique of sub- and super-solution (see \cite{smoller}), there exists a classical solution.

\end{proof}

\begin{remark}
If we relax the condition at $+\infty$, then the constant function $1$ is a trivial super-solution. We deduce the existence of a solution for any speed $c$. These solutions do not satisfy automatically the conditions at $+\infty$. The objective of this work is to understand on which conditions on $L$ and $c$, the limit of the solution is $0$ near $+\infty$. 
\end{remark}


\subsection{Study of the critical case $L=\Lambda(c)$}\label{subsec:critiq}

First, we prove that the tail at $+\infty$ is unique. Next, we use this uniqueness property to conclude all the intermediate remaining properties.

\subsubsection{Uniqueness of the tail}

\begin{lemma}\label{MainLemmaPS1}
Let $c$ be a fixed speed and two sizes $L_1$, $L_2$  be such that \eqref{PS1} with parameters $(c, L_{1,2})$ admits a solution $u_{1,2}$. Then, there exists $z_{+,-} \in \mathbb{R}$, such that
\[ \begin{aligned}
&u_1(x + z_-) = u_2(x) && \quad \text{ for } x < 0, \\
&u_1(x+ z_+) = u_2(x) && \quad \text{ for } x > \max(L_1, L_2).
\end{aligned} \]
\end{lemma}

\begin{proof}[Proof of Lemma \ref{MainLemmaPS1}]
We only prove that the tail at $+\infty$ is unique. The proof works the same for the other tail. We adopt the general strategy of the proof of Lemma 4.2.1 in chapter 4 of \cite{Cov}.\\
Let $u,v$ be two solutions of \eqref{PS1} with parameters $(c, L_1)$ and $(c, L_2)$. Let $x_0>L$ be such that $u,v(x)<\alpha_1$ for all $x>x_0$ (where $\alpha_1$ is introduced in \eqref{alpha12}). Without loss of generality, one can assume that $u(x_0)<v(x_0)$. According to section 2 of \cite{GuoMorita}, there exists $k,K, \lambda_-, \lambda_+>0$  such that 
\[ \begin{aligned}
&k e^{- \lambda_+x }\leq u,v (x) \leq Ke^{-\lambda_+x} && \text{ for }  x > L, \\
(\text{respectively } \quad &k e^{\lambda_-x } \leq 1-u,v(x) \leq Ke^{\lambda_-x} && \text{ for } x < 0).
\end{aligned}\]
We deduce the existence of $\tau>0$ such that $   v(x) < u(x-\tau)$ for all $x>x_1$ with $x_1\geq x_0$. We introduce 
\[ \tau_* = \inf \left\lbrace \tau>0, \quad v(x) < u(x - \tau)  \text{ for } x>x_1 \right\rbrace.\]
We claim that $u(\cdot - \tau_*) = v$. Assume by contradiction that it is not the case. It follows that $\inf \ u(\cdot - \tau_*) - v = 0$ and $v <  u(\cdot - \tau_*) $ (otherwise, there exists a contact point between $u(\cdot - \tau_*)$ and $v$ which implies by the maximum principle $u(\cdot - \tau_*) =  v$). We deduce that for any $\varepsilon>0$, there exists $\delta>0$ such that $v < u(\cdot - (\tau_* -\varepsilon)) + \delta $. Let $\delta, \varepsilon$ be a such couple which also satisfies 
\begin{equation}
    \varepsilon< \tau_* ,  \quad  v(x_1) < u(x_1- \tau_*+\varepsilon) \quad  \text{  and } \quad u(x_0) + \delta < \alpha_1.
\end{equation}
We introduce
\[ \delta_* = \inf \left\lbrace \delta>0, \quad v<  u ( \cdot -(\tau_* - \varepsilon) ) + \delta \text{ for } x>x_1\right\rbrace.\]
This infimum exists (since $\delta$ is an upper bound) and we claim that $\delta_* = 0$. Indeed, if $\delta_*>0$, we deduce that $\underset{ x \to + \infty}{\lim} u(x - \tau_* + \varepsilon) + \delta_* - v(x) = \delta_*>0$. Hence, by definition of $\delta_*$, we deduce that $\inf u(\cdot - \tau_* + \varepsilon) + \delta_* - v = 0  $ is reached at a point $x_2 \in ]x_1, +\infty[$. On the one hand, at this minimum point, we deduce that 
\[\begin{aligned}  
&(u(\cdot - \tau_* + \varepsilon) + \delta_* - v )'(x_2) = 0,  \quad - (u(\cdot - \tau_* + \varepsilon) + \delta_* - v )''(x_2)  \leq 0, \\
\text{ and } \quad&  0 < u(x_2 + \tau_* - \varepsilon) < v(x_2) < \alpha_1.
\end{aligned}\]
On the other hand, since $g$ is strictly decreasing in $]0, \alpha_1]$, we have in one hand
\[g(u(x_2+ \tau_* - \varepsilon)) - g(v(x_2))> 0\]
and in an other hand the following contradiction contradiction 
\[ -c (u(\cdot - \tau_* + \varepsilon) + \delta_* - v )'(x_2)  - (u(\cdot + \tau_* - \varepsilon) + \delta_* - v )''(x_2) = g(u(x_2+ \tau_* - \varepsilon)) - g(v(x_2))\]
We conclude that $\delta_*=0$ which implies the existence of $\varepsilon>0$ such that $u(\cdot - \tau_* + \varepsilon)<v$ for $x> x_1$. This is in contradiction with the definition of $\tau_*$.  
\end{proof}

\subsubsection{Monotonicity results} \label{monotonicityresults}

This section is devoted to the proof of the monotonicity of the solutions introduced in Definition \ref{intermediate3}.

\begin{proof}[Proof of Proposition \ref{intermediate5}]
We adopt the strategy of \cite{BerRodRyz}. We prove in a first step that any solution $u$ must be monotone in $\mathbb{R}_-$ and $(L, +\infty)$. The second step is to prove by contradiction that it is also monotone in $(0,L)$. Indeed, if it is not monotone, we may construct a super-solution bellow than the non-monotone solution $u$ and we conclude to the existence of a new non-trivial solution smaller than $u$.

\vspace{0.2cm}
\textbf{Step 1. }\textit{The solution $u$ is monotone in $\mathbb{R}_-$ and $(L, + \infty)$. } We will only prove the monotonicity in $\mathbb{R}_-$, the other part follows from similar arguments. We use the ideas of the proof of Lemma 3.6 (c) of \cite{BerRodRyz}. \\
Let $\delta>0$ be such that $u(0) < 1-\delta$. Since $\underset{x \to -\infty}{\lim} u(x)  = 1$, we deduce the existence of $R>0$ such that $1-\delta < u(x)$ for all $x<-R$. It follows that for $\tau$ large enough we have that 
\[u(0) \leq  u_\tau (0) := u(0-\tau) .\]
Since $u_\tau$ is also a solution of \eqref{PS1} restricted to $\mathbb{R}_-$, the maximum principle in unbounded domains implies that $u < u_\tau $. We define 
\[\tau_* := \inf \left\lbrace \tau > 0 , \quad u(x) < u_\tau(x), \ \forall x \in \mathbb{R}_- \right\rbrace.\]
It is clear that $\tau_* < R$. Next, it suffices to show that $\tau_* = 0$. Indeed, if $\tau_* = 0$, then for any $x<0$, $\tau > 0$, we have that $u(x) < u(x-\tau)$, i.e. $u$ is decreasing. By contradiction, assume that $\tau_*>0$ and let $\xi = \underset{x \in [-R, 0]}{\inf} (u_{\tau_*}(x) - u(x))$. We distinguish two cases:\\
\begin{itemize}
    \item \textbf{Case 1 :} \textit{$\xi>0$. } In this case, there exists $\varepsilon \in ]0, \tau_*[$ such that 
    \[ \underset{x \in [-R, 0]}{\inf} u_{\tau_*-\varepsilon}(x) - u(x) >0 .\]
    The maximum principle in unbounded domain gives that $u_{\tau_*-\varepsilon}(x) - u(x) \geq 0$ in $]-\infty, -R[$. It is in contradiction with the definition of $\tau_*$. 
    \item \textbf{Case 2 :} \textit{$\xi=0$. } In this case, by compactness of $[-R, 0]$, there exists $x_0 \in [-R, 0]$ such that $u_{\tau_*}(x_0) = u(x_0)$. The maximum principle implies that $u_{\tau_*}=u$. We introduce 
    \[ R_\delta = \inf \left\lbrace R>0, \quad u(x) < 1-\delta \quad \text{ for } \ x \geq -R \right\rbrace\]
    (remark that since $u(0) < 1 -\delta$, we have $R_\delta <+\infty$). In the same way, we define 
    \[R_\delta^* = \inf \left\lbrace R>0, \quad u_{\tau_*}(x) < 1-\delta \quad \text{ for } \quad x \geq -R \right\rbrace.\]
    Since  $u(x)>1-\delta$ for all $x < R_\delta$, it follows that $u_{\tau_*}(R_\delta) = u(R_\delta - \tau_*) > 1-\delta$ and thus $R_\delta^*>R_\delta$. However, this is in contradiction with the equality $u = u_{\tau_*}$. The claims holds true because

\end{itemize}

\vspace{0.2cm}

\textbf{Step 2.} \textit{The solution $u$ is decreasing in $(0, L)$. } Assume by contradiction that $u$ is not decreasing. We recall that by construction $u$, is the minimal solution greater than the trivial solution $0$.\\ 
First, we claim that for any solution, there exists a finite number of $x_n$ where $u'$ changes its sign. Indeed, if the claim is false, we deduce from the Rolle theorem that there exists $x_\infty \in [0,L]$ such that $u'(x_\infty)  = u''(x_\infty) = 0$. But this is impossible since
\[ -\mu u(x_\infty) = -cu'(x_\infty) - u''(x_\infty) = 0.\]
Therefore, we deduce that there exists a finite number of $(x_n)_{0 \leq n \leq N}$ such that $0 \leq x_0 \leq ... \leq x_N \leq L$ and $u'(x_n + \varepsilon)>0$ for any $\varepsilon$ small enough. We construct a non-trivial super-solution $\overline{u}$ from $u$ which satisfies $\overline{u} \leq u$. This concludes the proof of Proposition \ref{intermediate5} since it would imply the existence of a new non-trivial solution $\widetilde{u}$ such that $0 \leq \widetilde{u} \leq \overline{u} < u$. That would be in contradiction with the minimality of the solution $u$. \\

\vspace{0.2cm}

\textbf{Construction of $\overline{u}$. } Let $i \in \left\lbrace 0 , ..., N \right\rbrace$ be such that $u(x_i) = \underset{ k \in \left\lbrace 0, ..., N  \right\rbrace}{\inf} u(x_k)$. We distinguish 2 cases: $u(x_i) \geq u(L)$ and $u(x_i) < u(L)$.\\
\begin{itemize}
    \item \textbf{Case 1 :} $u(x_i) \geq u(L)$. We introduce 
    \[y = \sup \left\lbrace x \in \mathbb{R}, \quad u(y) = u(x_i) \right\rbrace.\]
    Then we deduce that $y \leq L$ since $u$ is decreasing in $(L, +\infty)$ and $u(y)\geq u(L)$. Next, we define piecewise $\overline{u}$:
    \[ \overline{u}(x)  = \left\lbrace\begin{aligned}&u(x) \text{ for } x \in ]-\infty, x_i[, \\ &u(x_i) \text{ for } x \in [x_i, y], \\ &u(x) \text{ for } x \in ]y, +\infty[. \end{aligned} \right. \]
    Since $G(\overline{u}(x) , x) = - \mu \overline{u}(x)<0$ for $x \in [x_i, y]$, we deduce that $-c \overline{u}'-\overline{u}''-G(\overline{u}, x) = 0$ if $x \in [x_i, y]^c$ and $-c \overline{u}'-\overline{u}''-G(\overline{u}, x)  = - G(\overline{u}(x), x) \geq  0$ if $x \in [x_i, y]$.
    \item \textbf{Case 2 :} $u(x_i) < u(L)$. As in the previous case, we introduce
    \[y = \sup \left\lbrace x \in \mathbb{R}, \quad u(y) = u(x_i) \right\rbrace.\]
    Then we deduce that $y > L$ since $u$ is decreasing in $(L, +\infty)$ and $u(x_i)\leq u(L)$. We define as above 
      \[ \overline{u}(x)  = \left\lbrace\begin{aligned}&u(x) \text{ for } x \in ]-\infty, x_i[, \\ &u(x_i) \text{ for } x \in [x_i, L], \\ &u(x-L+y) \text{ for } x \in ]L, +\infty[. \end{aligned} \right. \]
    As above, we deduce that for all $x \in [x_i, y]$, we have $G(\overline{u}(x), x) \leq 0$. We conclude as in the previous case. Since $y >L$, we have, thanks to the first part of the proof, that $u'(y)<0$. Then the function $\overline{u}$ satisfies 
    \[ \underset{ x \to L^-}{\lim} u'(x) = 0 \geq \underset{x \to L^+}{\lim} u'(x) = u'(y).\]
\end{itemize}
Therefore, $\overline{u}$ is a non-trivial super-solution. 
\end{proof}

\subsubsection{Study of $\Lambda(c)$}

\begin{proof}[Proof of Proposition \ref{intermdiate6}]
The function is well defined since there exists a solution for all $c<0$ thanks to Proposition \ref{intermediate1}. Next, we prove the assertions 1,2,3 and 4.

\vspace{0.2cm}

\textbf{Proof of assertion 1 :} \textit{Existence of an infimum $\Lambda_0>0$. } If such an infimum does not exist, we deduce that there exists a speed $c<0$ such that $\Lambda(c)<L_0$ where $L_0$ is the critical size mentioned in Theorem 2.6 in \cite{BerRodRyz}. Next, if we denote by $u$ the solution of \eqref{PS1} with parameters $(c,\frac{L_0}{2})$, it follows that $u$ is a super-solution of the following equation:
\begin{equation}\label{pbBer}
    \left\lbrace
    \begin{aligned}
    &-v'' = g(v)1_{]0, \frac{L}{2}[^c} - \mu v 1_{]0, \frac{L}{2}[} \quad \text{ for } x \in \mathbb{R}, \\
    &v(-\infty) = 1.
    \end{aligned}
    \right.
\end{equation}
Indeed, a direct computation gives $-u'' - G(u,x)  = c u' \geq 0$. This is in contradiction with \cite{BerRodRyz} (Theorem 2.6) because, from the existence of a such super-solution (which verifies $u(-\infty)=1$ and $u(\infty) = 0$), we would deduce the existence of a solution of \eqref{pbBer} which tends to $0$ as $x$ tends to $+\infty$. 

\vspace{0.2cm}

\textbf{Proof of assertion 2 :} \textit{$\Lambda$ is decreasing. } Take any speeds $c_1<c_2\leq0$. Let $L>\Lambda(c_1)$ be such that \eqref{PS1} with parameters $(c_1, L)$ admits a solution $u$. A straightforward computation shows that $u$ is a super-solution to \eqref{PS1} with parameters $(c_2, L)$. It follows that $\Lambda(c_2)<L$ for any $L>\Lambda(c_1)$. Passing to the lower limit, it follows 
\[ \Lambda(c_2) \leq \Lambda(c_1).\]

\vspace{0.2cm}

\textbf{Proof of assertion 3.}\textit{ $\underset{c \to -\infty}{\lim} \Lambda(c) = +\infty$. } We proceed by contradiction and assume that $\sup \Lambda(c) < \bar{\Lambda}$ for some positive constant $\bar{\Lambda}$. Next, let $u_c$ be the solution of \eqref{PS1} with parameters $(c, \bar{\Lambda})$. On one hand, with a similar analysis than in Proposition \ref{lemmasuper1}, we deduce that 
\[\begin{aligned}
&u_c(0) =\frac{ u_c(\bar{\Lambda})[ (|c| + \sqrt{c^2 + 4\mu}) e^{-\frac{(|c| - \sqrt{c^2 + 4\mu})\bar{\Lambda}}{2}} - (|c| - \sqrt{c^2 + 4\mu}) e^{-\frac{(|c| + \sqrt{c^2 + 4\mu})\bar{\Lambda}}{2}}]}{2\sqrt{c^2 + 4\mu}} \underset{c \to -\infty}{\longrightarrow}  u_c(\bar{\Lambda}) \\ 
&\text{ and } \quad | u_c'(0) | = \frac{8 \mu u_c(\bar{\Lambda})  e^{\frac{c \bar{\Lambda}}{2}} \sinh (\frac{\sqrt{c^2 + 4\mu} \bar{\Lambda}}{2})}{\sqrt{c^2 + 4\mu}} 
\underset{ c \to -\infty}{\longrightarrow } 0
\end{aligned}\]
where the latter limit holds because 
\[ e^{\frac{c \bar{\Lambda}}{2}} \sinh(\frac{\sqrt{c^2 + 4\mu} \bar{\Lambda}}{2}) \leq \frac 12 e^{\frac{c \bar{\Lambda}}{2} \left( 1 - \sqrt{1+ 4\frac{\mu}{c^2}} \right)} = \frac 12 \exp\left(\frac{ 2\mu \bar{\Lambda}}{c(1  + \sqrt{1+4\frac{\mu}{c^2}})}\right) .\]
On the other hand, since 
\[u_c'(\bar{\Lambda})^2 = -2 \int_0^{u_c(\bar{\Lambda})} g(u)du + c \int_{\bar{\Lambda}}^\infty u_c'(y)^2 dy, \]
it follows that $u_c(\bar{\Lambda}) \leq \beta $ (where $\beta$ is introduced in \eqref{alpha12}). We deduce that 
\[ u_c'(0)^2 = 2\int_{u_c(0)}^1 g(v)dv + 2|c|\int_{-\infty}^0 u_c'(x)^2 dx \geq  2  \underset{ u \in (0, \beta)}{\min }\int_{u}^1 g(v)dv > 0,\]
which implies the following contradiction 
\[ 0   <| u_c'(0)|  \underset{c\to -\infty}{\longrightarrow} 0.\]
\vspace{0.2cm}

\textbf{Proof of assertion 4.}\textit{ Study of the case $(c, \Lambda(c))$. } Let $u_L$ be a solution of \eqref{PS1} with parameters $(c, L)$. Let $\widetilde{v}$ be the solution of the following ODE
\[ \left\lbrace \begin{aligned}
&-c \widetilde{v}' - \widetilde{v}'' = g(\widetilde{v}), \\
&\widetilde{v}(L) = u_L(L), \quad \widetilde{v}'(L) = u_L'(L).
\end{aligned}\right.\]
We underline that $\widetilde{v}$ decreases in $]L, +\infty[$ and $\widetilde{v}(\infty) = 0$. If $\widetilde{v}$ is decreasing in the whole domain $\mathbb{R}$, we would deduce that either $\widetilde{v}$ diverges to $+\infty$ or $\widetilde{v}$ converges as $x \to -\infty$ to a zero of $g$ greater than $0$ : either $\alpha$ or $1$. Notice that the last case is impossible. Otherwise, we can construct a super-solution to \eqref{PS1} with parameters $(0, 0)$ by considering $\min(1, \widetilde{v}(\cdot - x_1))$ where $x_1$ is such that $\widetilde{v}(x_1) = 1$. This last super-solution is in contradiction with the main result of \cite{BerRodRyz}.
Then, it is clear that $\widetilde{v}(-\infty)< 1$ therefore, we have to consider the case $\widetilde{v}(-\infty) = \alpha$. If $\widetilde{v}(-\infty) = \alpha$, it would imply the existence of a traveling wave of speed $c$ connecting $0$ and $\alpha$, that is a  solution of 
\[ \left\lbrace \begin{aligned}
&- u_{\mathrm{TW}}'' - cu_{\mathrm{TW}}' = g(u_{\mathrm{TW}}), \\
&u(-\infty) = \alpha, \qquad u(+\infty) = 0.
\end{aligned} \right. \]
Therefore, we split the proof into two parts: the case where there does not exists a traveling wave $u_{\mathrm{TW}}$ solution of \eqref{TWalpha}, and the one where there exists such a traveling wave. 

\vspace{0.2cm}

\textbf{Part 1 : }\textit{ There does not exist a decreasing solution to \eqref{TWalpha}. } We deduce that $\widetilde{v}$ is not strictly decreasing in $\mathbb{R}$. It follows that $\widetilde{v}'$ changes its sign. We denote by $x_1=\sup\left\lbrace x < L, \quad \widetilde{v}'(x)  > 0 \right\rbrace$ 

Finally, we introduce  
\[ \varphi (x)  = \widetilde{v}(x + x_1).\]
We notice that $0$ is a local maximum  thus $\varphi'(0) = 0$ and $\varphi''(0)  <  0 $.
One has 
\begin{equation}\label{caractup0} 
\int_0^{\varphi(0)} g(u)du = c \int_{0}^\infty \varphi'(x)^2 dx.
\end{equation}
Notice that for $c<0$ since the right hand side of \eqref{caractup0} is strictly negative, we deduce that $\varphi(0) < \beta$. We also have that $\varphi''(0) < 0$ and we deduce that $\varphi(0) > \alpha$. From this function $\varphi$, we construct a \textit{minimal} solution. Let $u_L(x):= \varphi(x-L)$ be such that $u_L(L) = \varphi(0) := \varphi_0$. We notice that $u_L'(L) = 0$. From this function which satisfies \eqref{PS1} with parameters $(c,L)$ in $]L, +\infty[$, we construct a solution of \eqref{PS1} in the whole domain $\mathbb{R}$ by considering $L$ as a free parameter. First, we extend the solution in $[0,L]$:
\[ u_L(x)  =\varphi_0 e^{\frac{|c|(x-L)}{2}} \left[  \cosh\left(\frac{\sqrt{c^2 + 4 \mu}}{2}(x-L) \right) + \frac{c}{\sqrt{c^2 + 4\mu}} \sinh\left(\frac{\sqrt{c^2 + 4 \mu}}{2}(x-L)\right) \right].    \]
We deduce that 
\begin{align*}
    u_L(0)  =& \varphi_0 e^{\frac{-|c|L}{2}} \left[  \cosh\left(\frac{\sqrt{c^2 + 4 \mu}L}{2} \right) + \frac{|c|}{\sqrt{c^2 + 4\mu}} \sinh\left(\frac{\sqrt{c^2 + 4 \mu}L}{2}\right) \right], \\
    u_L'(0_+) =& \frac{c}{2} u(0) -\varphi_0 \frac{\sqrt{c^2 + 4 \mu}}{2} e^{\frac{-|c|L}{2}} \left[  \sinh\left(\frac{\sqrt{c^2 + 4 \mu}L }{2} \right) + \frac{|c|}{\sqrt{c^2 + 4\mu}} \cosh\left(\frac{\sqrt{c^2 + 4 \mu}L}{2}\right) \right] .
\end{align*}
Next, we extend the solution to $]-\infty, 0[$. According to the uniqueness (up to a translation) of the solution of \eqref{PS1} in $]-\infty, 0[$, we deduce that the solution is fully determined by the value $u_L(0)$. Moreover, the derivatives must match at $0$, therefore, we deduce that 
\begin{equation}\label{aboveeq} 
u_L'(0_+)  = - \sqrt{2 \int_{u_L(0)}^1 g(u)du + 2|c| \int_{-\infty}^0 u_{L}'(x)^2 dx}.
\end{equation}
The idea is to prove that there exists a size $L_*$ such that \eqref{aboveeq} has a solution. \\
On one hand, we notice that the left hand side is $0$ for $L=0$ whereas the right hand side is strictly negative. On the other hand, let $L_1$ be such that $u_{L_1}(0) = 1$ (such a size exists since $u_0(0) < 1$ and $u_L(0) \underset{ L  \to +\infty}{\rightarrow } + \infty$). By noticing that the left hand side verifies $\underset{ L \to L_1}{\lim} u_{L_1}'(0_+) < 0$ whereas the right hand side verifies $\underset{ L \to L_1}{\lim}  - \sqrt{2 \int_{u_L(0)}^1 g(u)du + 2|c| \int_{-\infty}^0 u_{L}'(x)^2 dx} =  0$, we conclude that a size $L_*$ which satisfies \eqref{aboveeq} does exist. \\
The case $c=0$ is already treated in \cite{BerRodRyz}.

Second, we claim that the above solution $u_*$ is the greatest solution of \eqref{PS1} with speed $c$: for any solution $u$ of \eqref{PS1} with parameter $(c, L)$ then we have 
\[ L_*\leq  L \quad \text{ and } \quad u \leq u_*\]
(with equality if and only if $L=L_*$). Indeed, from uniqueness of the tail (Lemma \ref{MainLemmaPS1}), we claim that there does not exist solutions bigger than the one above. Indeed, if such a solution $w$ of \eqref{PS1} exists with parameters $(c, L_*')$ such that $L_*' < L_*$, then, we deduce that $u_*(L_*) < w(L_*)$. Moreover by Lemma \ref{MainLemmaPS1}, there exists $r_1, \tau_1>0$ such that 
\[ \forall x > r_1, \ u_*(x)  = w(x - \tau_1) .\]
The Cauchy Lipschitz Theorem, implies that for all $x>L_*'$, $w(x) = \varphi(x-\tau_2)$ for some $\tau_2>0$. Since $\varphi''(0)<0$ and $L_*'<L_*$, we conclude that $w$ is increasing in an open subset included in $(L_*', L_*)$. It is in contradiction with Proposition \ref{intermediate5}. We conclude that $L_* = \Lambda(c)$ and this closes the proof of the first part. 

\vspace{0.25cm}

\textbf{Part 2 : }\textit{ There exists a decreasing traveling wave solution to \eqref{TWalpha}. }  The first step is to prove that there exists a size $L>0$ and some $x_0 \in \mathbb{R}$ such that \eqref{PS1} with parameters $(c, L)$ admits a solution which satisfies
\[u(L)  = u_{\mathrm{TW}}(x_0) \quad \text{ and } \quad u'(L) = u_{\mathrm{TW}}'(x_0). \]
It will follow that for any size $L$ such that \eqref{PS1} admits a solution $u_L$, then $u_L$ is comparable to $u_{\mathrm{TW}}$ in a neighbourhood of $+\infty$. \\
\textit{Construction of the solution. } Let $\kappa \in (0,1)$ and $x_\kappa \in \mathbb{R}$ be such that $u_{\mathrm{TW}}(x_\kappa) = \kappa \alpha$. We also define $\sigma_\kappa = u'_{TW}(x_\kappa)$. The goal is to find a size $L$ such that \eqref{PS1} with parameters $(c, L)$ admits a solution $u_L$ which verifies $u_L(L) = \kappa \alpha$ and $u_L'(L) = \sigma_{\kappa} $. If such a solution exists, it is defined by $u_{\mathrm{TW}}$ for $x>L$. Starting from $u_{\mathrm{TW}}$, we reconstruct $u_L$ in $]0, L[$: we have 
\[ \begin{aligned}
u_L(x) &= e^{\frac{|c|(x-L)}{2}} \left[ \kappa \alpha \cosh\left( \frac{\sqrt{c^2 + 4\mu}(x-L)}{2} \right) - \frac{(2 \sigma_\kappa + |c| \kappa \alpha )}{\sqrt{c^2 + 4\mu}} \sinh \left( \frac{\sqrt{c^2 + 4\mu}(x-L)}{2} \right) \right], \\
u_L'(x) & = \frac{|c|}{2}u_L(x) \\
&\  + \frac{\sqrt{c^2 + 4\mu}}{2}  e^{\frac{|c|(x-L)}{2}} \left[ \kappa \alpha \sinh\left( \frac{\sqrt{c^2 + 4\mu}(x-L)}{2} \right) - \frac{(2 \sigma_\kappa + |c| \kappa \alpha )}{\sqrt{c^2 + 4\mu}} \cosh \left( \frac{\sqrt{c^2 + 4\mu}(x-L)}{2} \right) \right].
\end{aligned}\]
The limit $x \to 0^+$ leads to 
\[ \begin{aligned}
u_L(0^+) &= e^{\frac{cL}{2}} \left[ \kappa \alpha \cosh\left( \frac{\sqrt{c^2 + 4\mu}L}{2} \right) + \frac{(2 \sigma_\kappa + |c| \kappa \alpha )}{\sqrt{c^2 + 4\mu}} \sinh \left( \frac{\sqrt{c^2 + 4\mu}L}{2} \right) \right], \\
u_L'(0^+) & =  \  e^{\frac{cL}{2}} \left[ \left(\sigma_\kappa ( \frac{|c|}{\sqrt{c^2 + 4 \mu}} - 1) - \frac{\kappa \alpha \mu}{\sqrt{c^2 + 4 \mu}} \right) \sinh\left( \frac{\sqrt{c^2 + 4\mu}L}{2} \right) -\sigma_\kappa  \cosh \left( \frac{\sqrt{c^2 + 4\mu}L}{2} \right) \right].
\end{aligned}\]
On the other hand, let $u_{-\infty}$ be the tail of any nonincreasing solution of \eqref{PS1} with a speed $c$. We deduce that the value $u_{-\infty}(0)$ determines $u_{-\infty}'(0^-)$ by the following identity:
\[ u_{-\infty}'(0^-) =  - \sqrt{2\int_{u_{-\infty}(0^-)}^1 g(u) du + 2  |c| \int_{-\infty}^0 u_{-\infty}'(x)^2 dx }. \]
Since if a solution exists, it is $C^1$ and it must satisfy
\[ u_{-\infty}(0^-)  = u_L(0^+) \quad \text{ and } \quad u_{-\infty}'(0^-) =  u_L'(0^+).\]
Therefore, we look for $L>0$ such that 
\begin{equation}\label{egderiv0}
    u_L'(0) = - \sqrt{2\int_{u_L(0)}^1 g(u) du + 2|c| \int_{-\infty}^0 u_{-\infty}'(x)^2 dx } .
\end{equation} 
On one hand, we claim that if $L=0$ then there holds
\[- \sqrt{2\int_{\kappa \alpha }^1 g(u) du + 2|c| \int_{-\infty}^0 u_{-\infty}'(x)^2 dx } < \sigma_\kappa < 0.\] Indeed, if it is not the case, there exists a point $x_0<0$ and $x_1 \in \mathbb{R}$ such that $u_{-\infty}(x_0) = u_{\mathrm{TW}}(x_0 + x_1)$ and $u_{-\infty}'(x_0) = u_{\mathrm{TW}}'(x_0 + x_1) $ (see the phase portrait in Figure \ref{figExplicationPreuve}). It would imply that $u_{-\infty}(-\infty) = \alpha$: a contradiction. 

\begin{figure}[h!]
\centering
\includegraphics[scale=0.65]{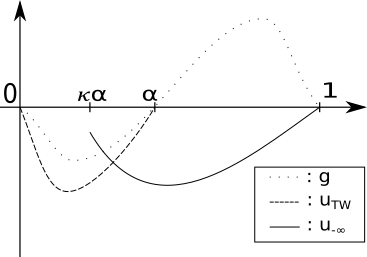}
\caption{Phase portrait of $u_{-\infty}$ and $u_{\mathrm{TW}}$. }
Since both solutions are decreasing, if $u_{-\infty}'(0)<u_{\mathrm{TW}}'(0)$ then both curves intersect : a contradiction.
\label{figExplicationPreuve}
\end{figure}

On the other hand, there exists $L_1>0$ such that $u_{L_1}(0) = 1$ (since $u_L(0) \underset{ L \to +\infty}{\longrightarrow} +\infty$). For a such size $L_1$, the right-hand side of \eqref{egderiv0} is $0$ whereas the left hand side is strictly negative.

We conclude that for any $\kappa \in (0,1)$, there exists a size $L_\kappa$ such that \eqref{PS1} with parameters $(c, L_\kappa)$ admits a solution which verifies 
\[ u_{L_\kappa} (L_\kappa) = \kappa \alpha \quad \text{ and } \quad u'(L_\kappa) = \sigma_\kappa.\]

\vspace{0.2cm}

We have proved that $\Lambda(c) > 0$ (Assertion 1). We have also proved that any solution behaves like $u_{\mathrm{TW}}$ in a neighbourhood of $+\infty$. It is clear that there does not exist a solution which satisfies $u(L)>\alpha$ (otherwise $u$ behaves like $u_{\mathrm{TW}}$ and $u(L) > \max u_{\mathrm{TW}}$ : a contradiction). We deduce that $\alpha = \underset{L > \Lambda(c)}{\sup } u(L)$. Moreover, since all the solutions are ordered (thanks to Proposition \ref{intermediate5}), we deduce that $L_\kappa \underset{ \kappa \to 1}{\rightarrow} \Lambda(c)$. It follows that 
\[  \underset{L \to  \Lambda(c), \ L>\Lambda(c)}{\lim  } u(L) = \alpha \quad \text{ and } \quad  \underset{L \to \Lambda(c), L>\Lambda(c)}{\lim } u'(L)=0.\]
\end{proof}

\begin{remark}
From the last part of the previous proof, we deduce that if there exists a traveling wave solution of \eqref{TWalpha}, then there exists a non-trivial minimal solution of 
\[\left\lbrace 
\begin{aligned}
&-cu' - u'' = g(u) 1_{\left\lbrace x < 0, \ x>\Lambda(c) \right\rbrace } - \mu u 1_{\left\lbrace 0 < x < \Lambda(c) \right\rbrace},\\
&u(-\infty) = 1
\end{aligned}
\right.\]
and satisfies $u(x) = \alpha$ for all $x\geq \Lambda(c)$. (here, by \textit{minimal} we mean that if there exists a solution $v$ of the above system such that $0 < v \leq u$ then $v=u$). 
\end{remark}

We conclude this subsection with the proof of Corollary \ref{Corollarygconv}.

\begin{proof}[Proof of Corollary \ref{Corollarygconv}]

It is a direct application of 3. Proposition \ref{intermdiate6} and Theorem 4.10 in \cite{Perthame}. This last result tells us that such a traveling wave exists if and only if $|c|\geq 2 \sqrt{g'(\alpha)}$. Indeed, since $g''(\alpha) > 0$, we deduce that this traveling wave exists and is solution of 
\[ \left\lbrace 
\begin{aligned}
&-c u_{\mathrm{KPP}}' - u_{\mathrm{KPP}}'' = g(u_{\mathrm{KPP}})\\
&u_{\mathrm{KPP}}(-\infty) = \alpha, \quad u_{\mathrm{KPP}}(+\infty)  = 0.
\end{aligned}\right.\]
The stable state $0$ invades the unstable state $\alpha$ with a speed $c$. Since $u_{\mathrm{KPP}} \in (0, \alpha)$, we deduce that the reaction term may be understood as a Fisher-KPP type reaction term.
\end{proof}

\subsubsection{Characterization of the case $L<\Lambda(c)$}

\begin{proof}[Proof of Proposition \ref{intermediate7}] Let $L<\Lambda(c)$ be such that the solution $u$ of 
\[\left\lbrace 
\begin{aligned}
&-cu' - u'' = g(u) 1_{\left\lbrace x < 0, \ x>\Lambda(c) \right\rbrace } - \mu u 1_{\left\lbrace 0 < x < \Lambda(c) \right\rbrace},\\
&u(-\infty) = 1
\end{aligned}
\right.\]
does not satisfies $u(\infty) = 0$. We will also denote by $u_{\infty}$ the solution of \eqref{PS1} with parameters $(c, \Lambda(c)+\varepsilon)$ with $\varepsilon$ a small parameter that we will fix later on.

First, we prove that $u$ is increasing in $[L, +\infty[$. Assume, by contradiction, that there exists $x \in [L, +\infty[$ such that $u'(x)< 0$. We define $x_0 = \inf \left\lbrace x > L, \quad u'(x)<0 \right\rbrace$. The point $x_0$ exists by hypothesis and satisfies $u(x_0) \geq u(L) > u_{\infty}(\Lambda(c)+ \varepsilon)$ according to Proposition \ref{intermediate5} (and for $\varepsilon$ small enough). Without lost of generality, we may assume $\varepsilon$ small enough such that $u'(x_0)<u_{\infty}'(\Lambda(c)+ \varepsilon)<0$. Next, we introduce $x_1 = \inf \left\lbrace x > x_0, \quad u'(x)>0 \right\rbrace$ such that $u$ is decreasing in $(x_0, x_1)$. We claim that $x_1 = +\infty$. \\
Indeed, if $x_1 < +\infty$, since $u(x_1)$ is a local minimum, we deduce that 
\[ u'(x_1) = 0, \quad g(u(x_1)) = -u''(x_1) < 0 \quad \Rightarrow \quad   u(x_1)\in(0,\alpha).\] 
Using again that $u'(x_1) = 0$, we conclude, thanks to a phase portrait analysis (see Figure \ref{LinfLambda}), the existence of $y \in ]x_0, x_1[$ and $z_0>\Lambda(c)$ such that $u(y) = u_\infty(z_0+y)$ and $u'(y) = u_{\infty}'(z_0+y)$. In this phase portrait, it is quite clear that if $u'<0$ only on an interval. The desired contradiction follows. The existence of this contact point between the two curves $(u, u')$ and $(u_{\infty}, u_{\infty}')$ implies the following contradiction : $u(x) \to 0$ as $x \to +\infty$. 

\begin{figure}[h!]
    \centering
    \includegraphics[scale=0.75]{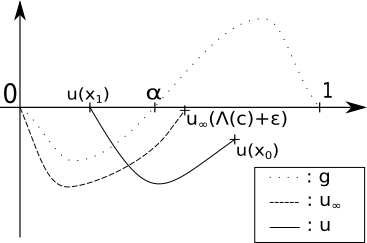}
    \caption{Existence of $y \in ]x_0, x_1[$ and $z_0>\Lambda(c)$ such that $u(y) = u_\infty(z_0+y)$ and $u'(y) = u_{\infty}'(z_0+y)$}
    \label{LinfLambda}
\end{figure}

 It follows that $u'<0$ in $(L, +\infty)$. Since $u$ is bounded, we deduce that $u$ converges either to $\alpha$ or to $0$. Since $L<\Lambda(c)$, $u$ converges to $\alpha$ (by definition of $\Lambda(c)$). Therefore, we have $u(x)>\alpha$ and $u'(x)<0$ for all $x>L$. We deduce that
\[ -u''(x) = c u'(x) + g(u(x)) > 0.\]
Thus, $u$ is a decreasing concave function which converges as $x$ tends to $+\infty$. This is a contradiction. We conclude that $u'(x) \geq 0$ for all $x\in [L, +\infty[$. 

\vspace{0.4cm}

Next, we prove that $u$ converges to $1$. The solution is bounded by the maximum principle and from the previous part of the proof, the solution is increasing for $x>L$. Hence, it converges to a $0$ of $g$. Since $u(L)>  \alpha$,  we conclude that $u$ converges to $1$. 

\vspace{0.4cm}

It remains to prove that $u'=0$ admits only one solution. If it admits two solutions, they are necessary in $(0,L)$. Then, one can construct a super-solution $\overline{u}$ of \eqref{PS1} with parameters $(c, L)$ such that $\overline{u}  \leq u$. Indeed, assume that $u'$ changes its sign at least twice. Since $u'(0)<0$ and $u'(L)>0$, we deduce that there exists $x_0, x_1 \in (0,L)$ such that 
\[x_0<x_1,  \quad u'(x_{0,1}) = 0 \quad \text{ and } \quad u''(x_{0,1})> 0.\]
Moreover, without loss  of generality we may assume that
\begin{enumerate}
    \item there does not exist any other point in $(x_0, x_1)$ that satisfies such conditions,
    \item $u(x_0) \geq u(x_1)$.
\end{enumerate}
We define $x_2 = \sup \left\lbrace x \in [x_0, x_1], \quad u(x)  = u(x_0) \right\rbrace$. According to the convexity of the curves, we deduce that $u'(x_2) \leq 0$ and $u(x)>u(x_0)$ for any $x\in (x_0, x_2)$. Next, we define 
\[ \overline{u}(x)  = \left\lbrace \begin{aligned}
&u(x) && \text{ for } x < x_0,  \\
&u(x_0) && \text{ for } x \in [x_0, x_2], \\
&u(x) && \text{ for } x >x_2.
\end{aligned}\right.\]
Following the same computations as in the proof of Proposition \ref{intermediate5}, $\overline{u}$ is a super-solution of \eqref{PS1}. Moreover, $\overline{u} \leq u$. This is in contradiction with the minimality of the solution $u$. Therefore, we deduce that $u$ changes its sign only once.

\end{proof}

\section{The sterile male release strategy}\label{SectionPS2}

This section is organized as follow: first we introduce some theoretical intermediate results and definitions. The next subsection is devoted to the proof of the properties of $m_S$. Indeed, the dynamics of the sterile males is independent of the one of the females (notice that the contrary is obviously false). Next, we prove the existence of a solution and the properties of this solution. 

\subsection{Intermediate results}

First, we introduce the main tool : a simple comparison principle for the system \eqref{PS2}. Since $(m_S \mapsto g(f, m_S) ) $ is decreasing, we conclude to the following comparison principle:
\begin{proposition}[Comparison principle for \eqref{PS2}]\label{comparisonPS2}
Let $(f_1, m_1)$ and $(f_2, m_2)$ satisfying
\[ 
\left\lbrace 
\begin{aligned}
&\partial_t f_1 -cf_1'-f_1'' - g(f_1, m_1) \leq 0, \\
& \partial_t m_1 - c m_1' - m_1'' - g_M (m_1 ) \geq 0, 
\end{aligned}
\right.
\quad \text{ and } \quad 
\left\lbrace\begin{aligned}
&\partial_t f_2 -cf_2'-f_2'' - g(f_2, m_2) \geq 0, \\
& \partial_t m_2 - c m_2' - m_2'' - g_M (m_2 ) \leq 0, 
\end{aligned} 
\right.\]
 and 
 \[ f_1(t=0) \leq f_2(t=0) \quad \text{ and } \quad m_2(t=0) \leq m_1(t=0).\]
 Then for all $t \geq 0$, we have
  \[ f_1(t) \leq f_2(t) \quad \text{ and } \quad m_2(t) \leq m_1(t).\]
\end{proposition}
It illustrates the intuition: the more sterile males are released, the less efficiently females reproduce. From this comparison principle, we construct solutions to \eqref{PS2} which satisfy $f(-\infty) = 1$ (the stable non-trivial equilibrium) and $f(+\infty) = 0$. 
We begin by providing a complete description of $m_S$. 
\begin{proposition}\label{Propo_Ms}
For any parameters $(L, c, \mathrm{M}) \in \mathbb{R}^{+*} \times \mathbb{R}^{-} \times \mathbb{R}^{+*}$, there exists a weak-solution $m_S$ to 
\begin{equation}\label{eqms}
    -c m_S' - m_S'' + \mu_s m_S = \mathrm{M} 1_{(0,L)}.
\end{equation}
Moreover, the following assertions hold true:
\begin{enumerate}
    \item $m_S \in H^1(\mathbb{R}) \cap C^1(\mathbb{R})$, 
    \item $0 \leq m_S \leq  \frac{\mathrm{M}}{\mu_s}$, 
    \item for all $x \in \mathbb{R}$
    \begin{equation}
    \label{Ms}
    m_S(x)  = \mathrm{M}L \int_{\mathbb{R}}  \frac{\sinc ( \frac{L\xi}{2}) e^{2i\pi \xi x +i\pi L}}{  2[ 4 \pi^2 \xi^2 +2 \pi i c \xi + \mu_s] } d\xi \qquad \left(\text{with} \quad \sinc(y) = \frac{\sin(y)}{y}\right),
    \end{equation}
    \item $\underset{ x \to \pm \infty}{\lim} m_S(x) = 0$, 
    \item there exists $x_0 \in (0,L)$ such that $m_S$ is increasing in $(-\infty, x_0) $ and decreasing in $(x_0, +\infty)$.
\end{enumerate}
\end{proposition}

Now, we have a full description of $m_S$, we state
\begin{proposition}\label{f}
For any $c\leq 0$ and $L>0$, there exists a parameter $\mathrm{M}$ such that $\eqref{PS2}$ with parameters $(c, \mathrm{M}, L)$ admits a solution $f$.
\end{proposition}
As for the killing strategy, the idea is to build a super-solution $\psi_+$ and a sub-solution $\psi_-$ that are well ordered ($\psi_- \leq \psi_+$). We will construct these sub and super-solutions such that for $\varepsilon$ small enough, there holds
\[ \psi_+(\infty) = 0, \quad \psi_{- | \mathbb{R}^+} = 0, \quad \psi_+(-\infty) = 1 \quad \text{ and } \quad \psi_-(-\infty) > 1-\varepsilon. \]

To play the an equivalent role as $\Lambda(c)$ in the previous problem, we introduce 
\[ \Pi (c, L) = \inf \left\lbrace \mathrm{M}>0, \quad \eqref{PS2} \text{ with parameters }(c, L, \mathrm{M}) \text{ admits a solution} \right\rbrace.\]
We prove some properties of $\Pi (c, L)$:
\begin{proposition}\label{Pi}
The function $\Pi$ is well defined. Moreover, for a fixed speed $c \leq 0$, the function $L \mapsto \Pi (c, L)$ is decreasing and satisfies 
\[ \underset{ L \to 0}{\lim } \  \Pi (c, L) = +\infty \quad \text{ and } \quad  \underset{ L \to +\infty}{\lim } \Pi (c, L) = \Pi_\infty(c) >0.\]
Finally, for any fixed size $L>0$, there holds
\[ \underset{  c \to -\infty}{\lim} \ \Pi(c,L) = +\infty.\]
\end{proposition}

To prove Proposition \ref{Pi}, we will use a corollary of Propositions \ref{comparisonPS2} and \ref{Propo_Ms}:
\begin{corollary}\label{ComparisonPrinciplePS22}
Let $(c, L_1, \mathrm{M}_1)$ and $(c, L_2, \mathrm{M}_2)$ be two sets of positive parameters such that \eqref{PS2} admits two solutions $(m_S^1, f_1)$ and $(m_S^2, f_2)$ and such that
\[ \mathrm{M}_1 \leq \mathrm{M_2} \quad \text{ and } \quad L_1 \leq L_2.\]
Then we have 
\[ m_S^1 \leq m_S^2 \quad \text{ and } \quad f_2 \leq f_1.\]
\end{corollary}

\subsection{Proof of Theorem \ref{mainPS2}}

Theorem \ref{mainPS2} is a direct application of Propositions \ref{f} and \ref{Pi}. 

\subsection{Study of the sterile male density}

\begin{proof}[Proof of Proposition \ref{Propo_Ms}]
We prove each of the points of Proposition \ref{Propo_Ms}. 

\textit{Proof of 1. } If we study the variational form associated to \eqref{eqms}, it follows
\[ \int_\mathbb{R} \left( (m_S')^2 + \mu_s m_S^2  -  \frac{c}{2} \left(m_S^2 \right)' \right) dx = \int_{0}^L \mathrm{M}  m_S dx. \]
The right hand side is a continuous application from $H^1(\mathbb{R})$ to $\mathbb{R}$. The bilinearity of the application 
\[ a(u,v) = \int_\mathbb{R} u' v' + \mu_s u v + \frac{c}{2} \left(uv \right)' dx \]
is straightforward. Since $H^1(\mathbb{R}) = H^1_0(\mathbb{R})$, we deduce that $\int_\mathbb{R} \left(u^2 \right)' dx = 0 $ and it follows
\[ \min(1, \mu_s) \| u \|_{H^1(\mathbb{R}) } \leq a(u,u).\]
We conclude, thanks to the Lax-Milgram Theorem, the existence of $m_S \in H^1(\mathbb{R})$ solution of \eqref{eqms} with parameters $(c, L, \mathrm{M})$. Moreover, by elliptic regularity, we deduce that $m_S \in C^1(\mathbb{R})$.

\vspace{0.4cm}

\textit{Proof of 2. } By remarking that $0$ is a solution and that $\frac{\mathrm{M}}{\mu_s}$ is a supersolution. We deduce thanks to the weak maximum principle that 
\[ 0 < m_S < \frac{\mathrm{M}}{\mu_s}.\]

\vspace{0.4cm}

\textit{Proof of 3. } According to \textit{1.}, we have $m_S \in L^2(\mathbb{R})$. If we perform a Fourier transform, it follows:
\[ 2 \pi i c \xi \hat{m}_s + 4 \pi^2 \xi^2 \hat{m}_s +\mu_s \hat{m}_s = \frac{\mathrm{M}L \sinc (\frac{ L \xi}{2}) e^{i\pi L} }{2}\]
where we have adopted the following definitions :
\[\hat{f}(\xi) = \int_\mathbb{R} f(x) e^{-2i\pi x\xi} dx .\]
We deduce that 
\[ \hat{m_S}(\xi) = \frac{\mathrm{M}L \sinc (\frac{ L \xi}{2}) e^{i\pi L}}{  2  [ 4 \pi^2 \xi^2 +2 \pi i c \xi + \mu_s] }.\]
Since $\underset{\xi \in \mathbb{R}}{\inf} | 4 \pi^2 \xi^2 + 2 \pi i c \xi + \mu_s| \geq \underset{\xi \in \mathbb{R}}{\inf} 4 \pi^2 \xi^2 + \mu_s>0$, we deduce that $\hat{m_S} \in L^1(\mathbb{R}) \cap L^2(\mathbb{R})$ and we conclude that
\[ m_S (z) =\mathrm{M}L\int_{\mathbb{R}}  \frac{ \sinc ( \frac{L \xi}{2 }) e^{2i\pi \xi z+ i\pi L}}{  2[ 4 \pi^2 \xi^2  - 2 \pi i c \xi + \mu_s] } d\xi.\]

\vspace{0.4cm}

\textit{Proof of 4. } By noticing that $\hat{m}_s \in L^1(\mathbb{R})$, we conclude that $m_S(x) \underset{ x \to \pm \infty}{\longrightarrow } 0$. 

\vspace{0.4cm}

\textit{Proof of 5. } Thanks to\textit{ 4.}, we deduce that $m_S$ takes its maximum at some point $x_0 \in \mathbb{R}$. Since the reaction term is positive only in $(0, L)$, the weak maximum principle implies that $x_0 \in (0,L)$.\\
Next, we prove by contradiction that $m_S$ is monotonous on $(-\infty, x_0)$ and on $(x_0, +\infty)$. If it is not the case in $
(-\infty, x_0)$ (the other case is treated with the same method), since $m_S(x) \underset{ x \to -\infty}{\rightarrow} 0$, we deduce that there exists $x_1 \in (-\infty, x_0)$ such that $m_S'(x_1)=0$ and $-m_S''(x_1) > 0$ (in a weak sense). Since $m_S(x_0) = \max m_S$ and $m_S'(x_1 + \varepsilon)<0$, it follows the existence of $x_2 \in (x_1, x_0)$ such that $m_S'(x_2) = 0$ and $-m_S''(x_2)<0$ (in a weak sense). Regarding the sign of the reaction term, is follows that $0<x_1 $ and $x_2<0$. Recalling  that $x_1 < x_2 < x_0$, it follows the desired contradiction. 
\end{proof}

\subsection{Construction of the solution $f$}

We adopt the same strategy to construct the sub and the super-solutions as for equation \eqref{PS1}. First, we begin with a super-solution constructed by parts, next, we focus on the subsolution. 

\subsubsection{Construction of a super-solution}

First, we fix a size $L>0$ and $c \leq 0$. We split the domain into three parts : $\mathbb{R}_-, (0, L_*)$ and $(L_*, +\infty)$ with $L_* \geq L$ an arbitrary size of interval. The density of released mosquitoes $\mathrm{M}$ at each position of the interval $(0, L)$ is a free parameter. Next, we introduce, for $u\in\mathbb{R}_+$
\begin{equation}
g_{super}(u, x) = \left\lbrace \begin{aligned}
&g(u, 0) && \text{ for } x \leq 0, \\
&-\frac{\mu}{2} u && \text{ for } 0 \leq x \leq L_* , \\
&g(u, 0) && \text{ for } \text{ for } x > L_*. 
\end{aligned}\right. 
\end{equation}
We claim the following result:
\begin{lemma}\label{Lemmasuperrealpb1}
For any $L>0$, there exists $\mathrm{M}_0>0$ such that for any $\mathrm{M}> \mathrm{M}_0$, we have $g(u,x) \leq  g_{super}(u,x)$ for all $x \in \mathbb{R}$ and $u\geq 0$. 
\end{lemma}

\begin{proof}
The assertions holds true for $x<0$ and $x>L_*$. For $x \in ]0, L_*[$, it suffices to remark that for a fixed value $L_*$, as $\mathrm{M}$ increases $\underset{x \in (0, L_*)}{\min} m_S(x)$ increases (see \eqref{Ms}). The conclusion follows by remarking that $g(u, x)$ converges locally uniformly to $-\mu u$ as $\mathrm{M} \to +\infty$. 
\end{proof}

The two following Lemmas are similar to Lemmas \ref{lemmasuper1} and \ref{Lemmasuper2}. The proofs are identical, therefore we do not provide them.

\begin{lemma}\label{Lemmasuperrealpb2}
There exists $\gamma_0 >0$ and $\delta_0>0$ such that the solution of the following ODE
\begin{equation}\label{PS1super1MS}
\left\lbrace
\begin{aligned}
&-c v_2' - v_2'' = g_{super}(v_2,x),     \\
&v_2(L_*) = \gamma_0, \quad v_2'(L_*) =-\delta_0
\end{aligned}
\right.
\end{equation}
satisfies 
\[v_2(+\infty)  = 0 \quad \text{ and } \quad v_2'(+\infty) = 0.\]
\end{lemma}

\begin{lemma}\label{lemmasuperrealpb3}
There exists a size $L_*>0$ such that there exists a solution $v_1$ of the following problem
\begin{equation}\label{eqsuper12}
\left\lbrace
\begin{aligned}
    &-c v_1' - v_1'' = - \frac{\mu}{2} v_1 \\
    &v_1(0) = 1, \quad v_1(L_*) = \gamma_0, \\
    &v_1'(0) \leq 0, \quad v_1'(L_*) =0.
\end{aligned}
\right.
\end{equation}
\end{lemma}

We conclude with the construction of the super-solution:
\begin{lemma}\label{SuperSolutionPS2}
For any speed $c \leq 0$, size $L>0$, there exists $\mathrm{M}>0$ such that \eqref{PS2} with parameters $(c, L, \mathrm{M})$ admits a super-solution $\psi_+$ which satisfies 
\[ \psi_+(-\infty) = 1 \quad \text{ and } \quad \psi_+(+\infty) = 0.\]
\end{lemma}
\begin{proof}
Let $c \leq 0$ be a fixed parameter. Let $\gamma_0$ and $\delta_0$ be given by Lemma \ref{Lemmasuperrealpb2}. Next, thanks to Lemma \ref{lemmasuperrealpb3}, we deduce the existence of $L_*$ such that the conclusion of Lemma \ref{lemmasuperrealpb3} holds true. Then, according to Lemma \ref{Lemmasuperrealpb1}, there exists a value of $\mathrm{M}$ large enough such that it is sufficient to construct a super-solution of the equation
\begin{equation}\label{eqsuperrealpb}
-c\psi_+' - \psi_+'' = g_{super}(\psi_+,x)
\end{equation}
to obtain a super-solution of $-cu' - u'' = g(u, m_S(x))$. The function
\[ \psi_+(x)= \left\lbrace\begin{aligned}
& 1&& \text{ for } x < 0, \\
& v_1(x) && \text{ for } 0 \leq x \leq L_*, \\
& v_2(x) && \text{ for } x >L_*, 
\end{aligned}\right. \]
(where the function $v_1$ and $v_2$ are provided by Lemmas \ref{Lemmasuperrealpb2} and \ref{lemmasuperrealpb3}) is a super-solution of \eqref{eqsuperrealpb}. Since $\underset{x \to -\infty }{\lim} \psi_+(x) = 1$ and $\underset{ x \to +\infty}{\lim} \psi_+(x) = \underset{ x \to +\infty}{\lim} v_2(x) = 0$, the conclusion follows. 
\end{proof}

\subsubsection{Construction of a sub-solution}

Since \eqref{PS2} is fully non-autonomous, the construction of the sub-solution is more tricky than the one for \eqref{PS1}. First, we introduce the following result:

\begin{lemma}\label{lemma:g(f,m)}
There exists $m_0>0$ such that for all $m\in [0, m_0]$, the application $(f \mapsto g(f,m))$ is bistable in the following sense:
\[ \begin{aligned} 
\exists \alpha_m, \varepsilon_m>0 \text{ such that }& \alpha_m < 1-\varepsilon_m,\\
&g(0, m) = g(\alpha_m, m) = g(1-\varepsilon_m, m) = 0,\\ 
&g(f, m) < 0 \quad \text{ for } 0 < f < \alpha_m,\\
\text{ and }& g(f,m)>0 \quad  \text{ for } \alpha_m < f < 1-\varepsilon_m.
\end{aligned}\]
Moreover, $m_0$ can be taken small enough such that $\alpha_m$ is an unstable equilibrium and $1-\varepsilon_m$ is a stable equilibrium with 
\[ \int_{0}^{1-\varepsilon_m} g(f, m) df > 0.\]
\end{lemma}

\begin{proof}Remarking that $(f \mapsto g(f, 0))$ is a bistable reaction term and noticing that $(m \mapsto g(F, m))$ is uniformly continuous with $g(0, m) = 0$ for any $m>0$ (assumption \eqref{H4PS2}), we deduce that there exists $m_0>0$, such that the conclusion is true.
\end{proof}

Remark that $\alpha_m \to \alpha$ and $1-\varepsilon_m \to 1$ as $m \to 0$. Therefore, without loss of generality we may assume $m_0$ small enough such that
\begin{equation}\label{gdecreasing}
\dfrac{\partial g   }{\partial f}(f, 0)<0 \quad \text{ for any } f \in ]1-\varepsilon_m, 1[.
\end{equation}
Let $m \in ]0, m_0[$ be fixed and $\psi_-$ be the solution of 
\begin{equation} \label{phimoins}
\left\lbrace
\begin{aligned}
&- \psi_-'' =g(\psi_-, m) && \text{ in } \mathbb{R}^-,\\
&\psi_-(0) = 0, \ \underset{ x < 0}{\underset{ x \to 0}{\lim}}\psi_-'(0) = -2\sqrt{\int_{0}^{1-\varepsilon_m} g(u, m)du},  \\
&\psi_- = 0 && \text{ in } \mathbb{R}^+.
\end{aligned}
\right.
\end{equation}
A classical portrait phase analysis tells us that $\psi_-(-\infty)  = 1-\varepsilon_m$ and $\psi_-' \leq  0$. We claim that a translation of $\psi_-$ is an admissible sub-solution. 

\begin{lemma}\label{subsolutionPS2}
There exists $x_0<0$ such that $ \phi_-(x) = \psi_-(x - x_0)$ is a sub-solution of \eqref{PS2}.
\end{lemma}

\begin{proof}
We define $x_0 =  \sup \left\lbrace x < 0,  \quad m_S(x)< m \right\rbrace$ (such a $x_0$ exists thanks to Proposition \ref{Propo_Ms}). Since $\phi_{-| \left\lbrace x > x_0 \right\rbrace}=0$, we focus on $x < x_0$:
\[ - c \psi_-' - \psi_-'' - g(\psi_-, m_S(x) ) = -c \psi_-' +g(\psi_-, m) - g(\psi_-, m_S(x)).\]
First we remark that $m_S(x) < m$ implies $g(\psi_-, m) - g(\psi_-, m_S(x)) < 0$ (thanks to \eqref{H4PS2}). Second, it is clear that $-c \psi_-'<0$. We conclude that
\[ - c \psi_-'-\psi_-'' - g(\psi_-, m_S(x) ) \leq 0 ,\]
Since the $\underset{ x < x_0}{\underset{ x \to x_0}{\lim }}\psi_-(x)' < 0  = \underset{ x > x_0}{\underset{ x \to x_0}{\lim }}\psi_-'(x) $, $\psi_-$ is a subsolution of \eqref{PS2}. 
\end{proof}

\subsubsection{Conclusion: Construction of a solution}

\begin{proof}[Proof of Proposition \ref{f}]

From the sub and the super-solutions established in Lemmas \ref{SuperSolutionPS2} and \ref{subsolutionPS2}, it follows the existence of a solution $f$ of \eqref{PS2} with parameters $(c, L, \mathrm{M})$ which satisfies:
\[ \underset{ x \to +\infty}{\lim} f(x) = 0 \quad \text{ and } \quad  1-\varepsilon_m \leq \underset{ x \to -\infty}{\lim} f(x) \leq  1.\]
We claim that $\underset{ x \to -\infty}{\lim} f(x)  = 1$. We distinguish two cases : 
\[     \underset{ x \to -\infty}{\liminf} f(x) =  \underset{ x \to -\infty}{\limsup} f(x) \quad \text{ or } \quad \underset{ x \to -\infty}{\liminf} f(x)  <   \underset{ x \to -\infty}{\limsup} f(x) .\]

\vspace{0.2cm}

\begin{description}
\item [Case 1.] In this case, if $f$ converges, it is to a zero of $g( \cdot, 0)$. Moreover, this root must belong to $[1-\varepsilon_m, 1]$. It follows that this root is $1$. We conclude that 
\[ \underset{ x \to -\infty}{\lim} f(x) = 1.\]
\item [Case 2.] In this case, we deduce that $f'$ changes its sign in a neighbourhood of $-\infty$. Let $\left( y_n \right)_{n \in \mathbb{N}}$ be a decreasing sequence such that $y_n \to -\infty$, $f(y_n)$ is a local minimum of $f$ and $|f(y_n) - \underset{x \to -\infty}{\limsup}f(x)| >\frac{1}{2} |\underset{ x \to -\infty}{\liminf} f(x) - \underset{x \to -\infty}{\limsup}f(x)| := \delta $. Since $f(y_n)$ is a local minimum, we deduce that 
\[ f'(y_n) = 0 \quad \text{ and } \quad g(f(y_n), m_S(y_n)) = -f''(y_n) \leq 0.\]
Moreover, since $1-\varepsilon_m \leq \liminf f(y_n) \leq 1 - \delta$, we deduce that 
\[\liminf \  g(f(y_n), m_S(y_n)) = \liminf  \ g(f(y_n), 0) > 0 .\]
We have a contradiction, this case is impossible. 
\end{description}

\vspace{0.4cm}

We conclude that $f(x) \underset{ x \to -\infty}{\longrightarrow} 1$ and the solution verifies all the desired properties. 

\end{proof}

\subsection{Study of $\Pi(L, c)$}

\begin{proof}[Proof of Proposition \ref{Pi}]We split the proof into 3 parts: first we prove that $\Pi$ is well defined, next we prove that it is decreasing, finally, we study the limit $L \to +\infty$. 

\vspace{0.4cm}

\textit{Proof that $\Pi$ is well defined. } Since for any size $L>0$ and speed $c \leq 0$, a solution to \eqref{PS2} exists for $\mathrm{M}$ large enough, we deduce that $\Pi(c,L)$ is well defined. 

\vspace{0.4cm}

\textit{Proof that $\Pi$ is decreasing with respect to $L$. } Let two sizes of interval $L_1$ and $L_2$ be such that $L_1 < L_2$. Let $\mathrm{M}>0$ be such that \eqref{PS2} with parameters $(c, L_1, \mathrm{M})$ admits a solution $f^1$. Finally, we introduce $m_S^{1}$ and $m_S^2$ as the two associated (with respect to $L_i$) distributions of sterile males. We claim that $f^1$ is a super-solution to \eqref{PS2} with parameters $(c, L_2, \mathrm{M})$. Indeed, since $L_2> L_1$, we deduce thanks to \eqref{Ms} that $m_S^1 < m_S^2$. By recalling that $m \mapsto g(f, m)$ is decreasing, it follows that 
\[ - c (f^1)' - (f^1)'' - g(f^1, m_{S}^{2})  = g(f^1, m_{S}^{1}) - g(f^1 , m_{S}^{2})>0.\]
By the method of the sub and super-solution presented above, we deduce that \eqref{PS2} with parameters $(c, L_2, \mathrm{M})$ admits a solution. Passing to the infimum, we conclude that 
\[ \Pi (c, L_1) \geq \Pi (c, L_2). \]
 
 \vspace{0.4cm}

\textit{Proof that $\underset{L \to 0}{\lim} \ \Pi (c, L) = +\infty$. } Assume by contradiction that it is not the case. Let $\Pi_0$ be the supremum of $\Pi$ in a neighborhood of $L=0$. Let $m_S^L$ be the distribution of sterile males associated with parameters $(c, L, \Pi_0)$ and $f_L$ the associated solution of \eqref{PS2}. We introduce $m_0$ such that $\left(f \mapsto g(f, m_0)\right)$ is bistable and $\mathrm{sign} \left( \int_0^{1-\varepsilon_{m_0}} g(u, m_0)du \right) = \mathrm{sign} \left( \int_0^{1-\varepsilon_{m_0}} g(u , 0)du \right) > 0$. According to \eqref{Ms}, we have $m_S^L(x) \underset{ L \to 0}{\longrightarrow} 0$ uniformly therefore, we assume that $L$ is small enough such that $m_S^L<m_0$. We claim that $f_L$ is a super-solution of
\begin{equation}\label{gum0}
-cu' -u''   = g(u, m_0).
\end{equation}
Indeed, according to \eqref{H4PS2}, it follows that
\[ -cf_L' - f_L'' - g(f_L, m_0) = g(f_L, m_S ) - g(f_L , m_0) > 0.\]
Hence, thanks to Lemma \ref{subsolutionPS2}, we deduce the existence of a traveling wave solution of \eqref{gum0} which connects the two stable states $0$ and $1-\varepsilon_{m_0}$. Moreover, this traveling wave has a negative speed $c$ which is in contradiction with \cite{aronsonweinber} (since $\int_0^{1-\varepsilon_{m_0}} g(u, m_0)du>0$ by hypothesis). 

\vspace{0.4cm}

\textit{Proof that  $\underset{L \to +\infty}{\lim}\Pi (c, L) = \Pi_\infty(c)>0$. } The proof works the same as the above : By contradiction, if we assume $\underset{L \to +\infty}{\lim}\Pi (c, L) = 0$, one can deduce the existence of a traveling wave of a bistable equation with a negative speed which connects the two stable states. Moreover, this bistable equation can be chosen such that the only traveling wave that connects these two stable states has a positive speed (see \cite{aronsonweinber}): a contradiction. We let the details to the reader. 

\vspace{0.4cm}

\textit{Proof that $\underset{c \to -\infty}{\lim} \Pi(c,L) = +\infty.$ } The proof works the same as the above proof : by contradiction. If the conclusion is false for a size $L>0$, it follows that there exists $\mathrm{M}_0>0$ such that for any negative speed $c$, $\Pi(c,L)< \mathrm{M}_0$. Next, the dominated convergence theorem applied to \eqref{Ms} implies that $\| m_S \|_\infty \to 0$ as $c\to -\infty$. Let $m_0$ be the parameter provided by Lemma \ref{lemma:g(f,m)}, $c<0$ such that $\Pi(c, L) < \mathrm{M}_0$ and $m_S < m_0$. We conclude by remarking that \eqref{PS2} with parameters $(c, L, \mathrm{M}_0)$ admits a solution $u$ and such a solution is a super-solution of
\[ -cf' - f'' = g(f, m_0).\]
This gives the existence of a traveling wave solution with a negative speed to the equation 
\[ \left\lbrace
\begin{aligned}
&\partial_t f - f'' = g(f, m_0), \\
&f(-\infty) = 1-\varepsilon_{m_0}, \quad f(+\infty) = 0.
\end{aligned}
\right.\]
By taking $m_0$ small enough such that $\int_0^{1-\varepsilon_{m_0}} g(u, m_0)du>0$ the contradiction follows. 
\end{proof}

\section{Numerical illustrations}\label{SectionNum}

First, we detail the numerical schemes. We discretize a parabolic version of \eqref{P'} and we let the time increases until the numerical solution reaches a numerical equilibrium. This equilibrium is assumed to be achieved if the error between two large times is small enough (i.e. a Cauchy criteria). We use an interval large enough to consider that the derivatives of the stable states at the boundaries are $0$. Therefore, we implement \eqref{P'} with Neumann boundary conditions:
\begin{equation}\label{eq:time:dpt}
\left\lbrace 
\begin{aligned}
    &\partial_t u - c \partial_x u - \partial_{xx} u = g(u) 1_{(0, L)} + Act(x, u)1_{(0, L)^c},\\
    &\partial_{\nu_x} u = 0.
\end{aligned}
\right.
\end{equation}

In a first subsection, we present a numerical illustration of Theorem \ref{mainPS1} for the function 
\[ g(u) = u(1-u)(u-\alpha).\]
This classical bistable reaction term verifies \eqref{H1S1}, \eqref{H2S1} and \textit{$g$ is convex in $(0, \alpha)$}. For the computation of $\Lambda(c)$ for \eqref{PS1}, we use a dichotomy method : We fix a speed $c$ and let two sizes of intervals $L_-<L_+$ be such that \eqref{PS1} admits a solution for parameters $(c, L_+)$ and does not admit a solution for parameter $(c, L_-)$. Next, we solve numerically \eqref{PS1} with parameters $(c, \frac{L_- + L_+}{2})$. If the solution invades the territory, we replace $L_+$ by $\frac{L_- + L_+}{2}$ and if the solution does not invade the territory, we replace $L_-$ by $\frac{L_- + L_+}{2}$. We underline that the sign of the derivative at $L$ provides a suitable invasion criterion.

For the second strategy, we provide a simulation of \eqref{eq:time:dpt} with $g$ which satisfies \eqref{gms} with different sets of parameters $(c,L, \mathrm{M})$. Since we do not have a precise invasion criterion as above, we only provide the critical size that ensures invasion for a fixed speed $c$ and a constant number of released sterile mosquitoes $\mathrm{M}$. 

\subsection{The killing strategy for the "classical" bistable reaction term}

We perform the numerical investigation of $\Lambda(c)$ for different values of $c$ between $c_\mathrm{min}$ and $c_\mathrm{max}$. The numerical parameters are fixed as follow : 
\begin{center}
\begin{tabular}{|c|c|c|c|c|c|c|c|c|c|}
\hline
   $x_{\min}$ & $x_{\max}$ & $L_{\min}$ & $L_{\max}$ & $\alpha$  & $dx$ & $dt$ & $c_{\min}$ & $c_{\max}$ \\
   \hline
   -75 & 75 & 2 & 39 & $\frac{1}{4}$  & 0.03 & 0.5 & 0 & 2.2 \\
   \hline
\end{tabular}
\end{center}

Figure \ref{Lambda(c)} represents the dependence of $\Lambda(c)$ with respect to |c|. Figure \ref{c1_alpha...} presents the numerical simulations of \eqref{PS1} with parameters $(1, \Lambda(1))$ (for the orange curve) and $(1, \Lambda(1) - 10^{-4})$ (for the blue curve). Notice that the value $\Lambda(c)$ used here is the one presented in Figure \ref{Lambda(c)}. We recover that the \eqref{PS1} with parameters $(c, L)$ admits a solution if and only if $L>\Lambda(c)$. Finally, Figure \ref{g(u(Lambda))} presents the numerical computations of $g(u(\Lambda(c))$ (with $\Lambda(c)$ obtained numerically). We recover that for $L$ close to $\Lambda(c)$, we have $u(\Lambda(c))$ close to $\alpha$ for $|c| \geq  \frac{\sqrt{3}}{2}$.

\begin{figure}[h]
\begin{minipage}{0.3 \linewidth}
    \centering
    \includegraphics[scale=0.32]{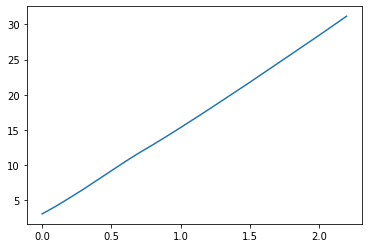}
    \caption{Numerical computation of $\Lambda(c)$ with respect to $|c|$}
    \label{Lambda(c)}
\end{minipage}
\hfill
\begin{minipage}{0.3 \linewidth}
    \centering
    \includegraphics[scale=0.32]{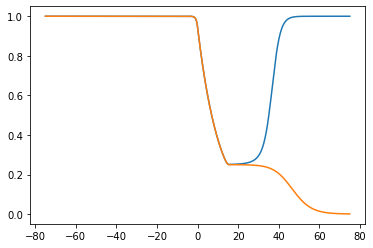}
    \caption{Numerical results of \eqref{PS1} with parameters $(1, \Lambda(1))$ (orange curve) and $(1, \Lambda(1) - 10^{-4})$ (blue curve) }
    \label{c1_alpha...}
\end{minipage}
\hfill
\begin{minipage}{0.27 \linewidth}
    \centering
    \includegraphics[scale=0.32]{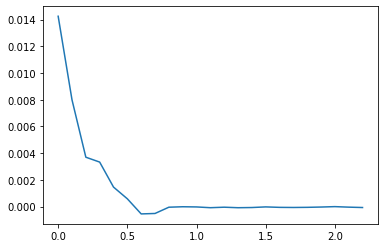}
    \caption{Numerical computation of $g(u(\Lambda(c))$ with respect to $|c|$ }
    \label{g(u(Lambda))}
\end{minipage}
\end{figure}

\subsection{A biological investigation : \textit{the Aedes albopictus} }

For the numerical illustration of this strategy, we present a numerical simulation of the parabolic system modeling specifically the mosquitoes case introduced in \cite{strugarek}. First, we recall the system introduced in Chapter 9 of \cite{strugarek} and the spatially version developed in \cite{almeida2020sterile}. The considered system reads 
\begin{equation}\label{numeqmos}
    \left\lbrace 
    \begin{aligned}
    &\partial_t f - D \partial_{xx} f = \frac{r \nu_E K b \tau f^2(1-e^{\tau f  + \gamma_s m_S})}{b \tau f^2(1-e^{\tau f  + \gamma_s m_S}) + K (\nu_E + \mu_E)(\tau f + \gamma_s m_S)} - \mu_F f\\
    &\partial_t m_S - D  \partial_{xx} m_S = \mathrm{M}1_{(-ct, L-ct)} - \mu_S m_S,  \\
    &f(t=0) = F 1_{\left\lbrace x < a\right\rbrace}, \quad m_S = \mathrm{M}1_{(0,L)},
    \end{aligned}
    \right.
\end{equation}
where $a$ is a positive constant and $F$ is the non-trivial stable equilibrium of $g(\cdot, 0)$. We refer to \cite{strugarek, almeida2020sterile} for a derivation of the model and the explanation of the meaning of the constants. But we provide typical values of the constants:
\begin{center}
\begin{tabular}{|c|c|c|c|c|c|c|c|c|}
\hline
   $r$ & $\nu_E$ & $K$ & $b$ & $\tau$  & $\gamma_s$ & $\mu_F$ & $\mu_S$  \\
   \hline
   0.49 & 0.7 & 1440 & 10 & 0.41  & 1 & 0.04 & 0.1 \\
   \hline
\end{tabular}
\end{center}
For such a choice of parameters, the reaction term $g(\cdot, 0)$ is bistable and one can apply Theorem \ref{mainPS2}. Moreover, we notice that the equilibria $F$ (the non-trivial stable one) and $F'$ (the unstable one) of $g(\cdot, 0)$ are such that 
\[ 0 < F' << F. \]
The aim of the numerical investigation is to illustrate the fact that the generated traveling wave eradicates the \textit{natural} invasive mosquitoes wave. As above, we solve numerically \eqref{numeqmos} by a semi-implicit scheme in a large space interval (large enough to consider that the boundary conditions are of Neumann types). For a fixed number of released mosquitoes ($\mathrm{M} = 20000$) and a fixed speed $c$ ($c=-0.05$), we manage to compute numerically the optimal size of released by a dichotomy method. We present this result in Figure \ref{FigTWSteril}. Indeed, Figure \ref{FigTWSteril} represents the number of females depending on position $x$ and time $t$. The red lines represent the moving interval where we release the mosquitoes : $(-ct, L-ct)$. We observe that in the first case ($L=17.4$), the interval is too small and the population does not go to extinction whereas in the second case ($L=17.49$), the population of females goes to extinction. In this second case, we succeed in generating a traveling wave solution of \eqref{PS2}.
\begin{figure}
\begin{minipage}{0.45 \linewidth}
    \centering
    \includegraphics[scale=0.45]{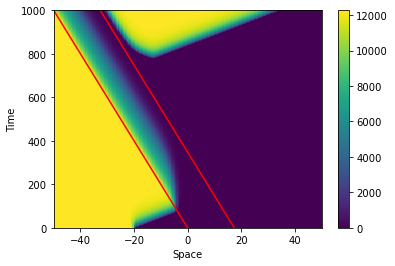}
\end{minipage}
\hfill
\begin{minipage}{0.45 \linewidth}
\centering
    \includegraphics[scale=0.45]{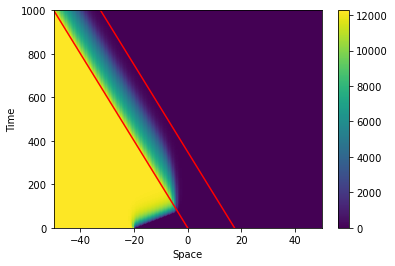}
\end{minipage}
\caption{Numerical simulations of the proportion of females subject to the dynamics of \eqref{numeqmos} with parameters $\mathrm{M} = 20000$, $c=-0.05$ varying the size $L$ of the released interval : Left, $L=17.4$ ; Right, $L=17.49$. Using the dichotomy method, we have obtained the minimal size of interval that ensures invasion for a such set of parameters $(c, \mathrm{M})$. }
\label{FigTWSteril}
\end{figure}

\section{Conclusions and perspectives}\label{SectionConclusion}
 
In conclusion, we study the existence of a traveling wave with negative speed to eradicate a population for both strategies. For the first strategy, this traveling wave exists if and only if the size of the interval where we act is large enough. We also manage to have a full description of the critical case. From a practical point of view, this is an interesting result. Indeed, since killing individuals has a cost, by maximizing the speed, we minimize the duration of the treatment and therefore we minimize the cost of the strategy. \\
For the sterile male strategy, the picture is different. For any size of interval, if we release enough sterile males $\mathrm{M}$, there exists a traveling wave solution. For this strategy, its cost may be defined by the quantity $\mathrm{N}$ of sterile males needed to be produced/released at each time:
\[ \mathrm{N}  = L \times \mathrm{M}.\]
Therefore, a perspective of this work is to optimize this number: giving a speed $c \leq 0$, how to minimize $N$. The aim is to find $\Lambda(c)$ that satisfies
\[ \underset{L>0}{\min } \ L \times \Pi(c, L) = \Lambda(c) \times \Pi(c, \Lambda(c)).\]

\vspace{0.4cm}

For the same purpose (i.e. minimize the cost of production), an other idea is to consider heterogeneous distributions of released of sterile males. For instance, we can model the release of sterile males by a piecewise constant function. The question is \textit{can we generate an eradication traveling wave as above but releasing less mosquitoes ?}\\
A first numerical investigation provides a positive answer. Indeed, we can take advantage that we expect $f(x=L) < f(x=0)$. Therefore, it is not necessary to release as many individuals in the area $x$ close to $L$ than in the area $x$ close to $0$. This is why, we compare the numerical simulations provided after the two following distributions of release:
\[ \mathrm{M} 1_{(0, L_*)}(x) \qquad \text{ and } \qquad   \mathrm{M} 1_{(0, \frac{2L_*}{3})}(x) + \frac{\mathrm{M}}{2} 1_{(\frac{2L_*}{3}, \frac{7L_*}{6})}(x)\]
(where $L_*$ is the minimal numerical size that ensures the extinction of the population for $\mathrm{M}$ and a speed $c$ introduced in section \ref{SectionNum}). The two numerical results are presented in Figure \ref{FigcomparisonStrat}. In both cases, we succeed in eradicating the invasive species. 
\begin{figure}
\begin{minipage}{0.45 \linewidth}
    \centering
    \includegraphics[scale=0.45]{a1749.png}
    Homogeneous released
\end{minipage}
\hfill
\begin{minipage}{0.45 \linewidth}
\centering
    \includegraphics[scale=0.45]{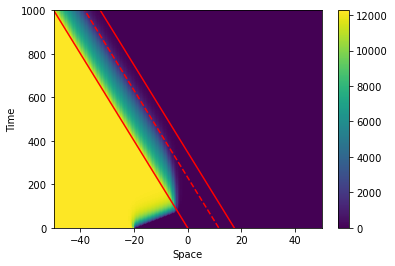}
    Heterogeneous released
\end{minipage}
\caption{Comparison of homogeneous released with a heterogeneous released that need less sterile males. In the heterogeneous case, the first strip represents the moving interval $(-ct, \frac{2L_*}{3}-ct)$ and the second strip stands for $(\frac{2L_*}{3}-ct, \frac{7L_*}{6} - ct)$.}
\label{FigcomparisonStrat}
\label{FigTWSteril2}
\end{figure}

If we compute the number of released sterile males at each step for each strategy, it follows
\[ \mathrm{N}_{homogeneous} = \mathrm{M} L_* \quad \text{ and } \quad \mathrm{N}_{heterogeneous} = \frac{11\mathrm{M} L_*}{12}.\]  
We conclude that less sterile males are needed for the heterogeneous releases. A deeper study of this strategy can be a good approach in order to minimize the total cost of production of sterile males. 

\vspace{0.4cm}

Another perspective is to focus on a problem which is closer to the reality: the same problem in a domain $\Omega \subset \mathbb{R}^2$. Without specifying it, we have used plenty of times that localized area where we \textit{act} separates $\mathbb{R}$ into two disconnected areas. It is obvious that it can be false for higher dimension. Therefore, we expect new difficulties coming from this fact.

\vspace{1cm}

\noindent \textbf{Acknowledgements. } The authors thank B. Perthame for all the fruitful discussions and his precious advices. \\
AL has received funding from the European Research Council (ERC) under the European Union's Horizon 2020 research and innovation program (grant agreement No 740623).

\bibliographystyle{plain} 
{\footnotesize
\bibliography{Biblio}}
\end{document}